\documentclass{article}

\usepackage[english]{babel}
\usepackage[utf8]{inputenc}
\usepackage[T1]{fontenc}
\usepackage{ae,lmodern}
\usepackage[margin=1.0in]{geometry}
\usepackage{amsmath}
\usepackage{amssymb}
\usepackage{amsthm} 
\usepackage{float} 
\usepackage{subcaption} 
\usepackage[group-separator={,}]{siunitx} 
\usepackage{multicol}
\usepackage{multirow}
\usepackage{makecell}
\usepackage{soulutf8} 
\usepackage{dsfont} 
\usepackage{pifont} 
\usepackage{tikz} 
\usepackage{tikzsymbols}
\usetikzlibrary{calc}
\usepackage{booktabs} 
\usepackage{pgfplots} 
\pgfplotsset{compat=1.17}
\usepackage{pgfplotstable} 
\usepackage{algorithm}
\usepackage{algpseudocode} 
\usepackage{todonotes} 
\usepackage{empheq} 
\usepackage{cancel} 
\usepackage{textgreek} 
\usepackage{xfp} 
\usepackage{nicefrac}
\usepackage{cite} %
\usepackage[colorlinks = true,
            linkcolor = blue,
            urlcolor  = blue,
            citecolor = blue,
            anchorcolor = blue]{hyperref} 
\usepackage[capitalise,nameinlink,noabbrev]{cleveref} 

\newcommand*{\email}[1]{%
    \href{mailto:#1}{#1}\par
    }

\newcommand{\eqbydef}{\mathrel{\mathop:}=}

\usetikzlibrary{calc}
\DeclareDocumentCommand{\hcancel}{mO{0pt}O{0pt}O{0pt}O{0pt}}{%
    \tikz[baseline=(tocancel.base)]{
        \node[inner sep=0pt,outer sep=0pt] (tocancel) {\tiny#1};
        \draw ($(tocancel.south west)+(#2,#3)$) -- ($(tocancel.north east)+(#4,#5)$);
    }%
}%

\newcommand{\brokenPolyT}[2]{\mathbb{P}^{#2}(\mathcal{T}_{#1})}
\newcommand{\brokenPolyF}[2]{\mathbb{P}^{#2}(\mathcal{F}_{#1})}

\newcommand{\traceSpace}[0]{H^{\nicefrac{1}{2}}(\partial\Omega)}
\newcommand{\normalderSpace}[0]{H^{\nicefrac{-1}{2}}(\partial\Omega)}

\newcommand{\hybridSpace}[0]{\underline{U}_h}
\newcommand{\hybridSpaceDirich}[1]{\underline{U}_{h,#1}}
\newcommand{\locHybridSpace}[0]{\underline{U}_T}

\newtheorem{theorem}{Theorem}

\newtheorem{lemma}{Lemma}
\newtheorem{remark}{Remark}

\newcommand{\cells}[0]{\mathcal{T}_h}

\newcommand{\bcells}[0]{\mathcal{T}_h^{\mathrm{B}}}
\newcommand{\facesT}[0]{\mathcal{F}_T}
\newcommand{\faces}[0]{\mathcal{F}_h}
\newcommand{\ifaces}[0]{\mathcal{F}_h^{\mathrm{I}}}
\newcommand{\bfaces}[0]{\mathcal{F}_h^{\mathrm{B}}}

\newcommand{\rcvT}[0]{\Theta_T}

\newcommand{\rcvtf}[0]{\Theta_{\cells}}

\newcommand{\homogC}[2]{\mathring{#1}({#2})}
\newcommand{\homog}[3]{\mathring{#1}_{#2}^{#3}}

\newcommand{\normalder}[1]{\partial_{\mathrm{n}} #1}
\newcommand{\flux}[0]{\partial_{\mathrm{n},h}}

\newcommand{\stabcellIP}[1]{(#1)^{\star}}

\definecolor{forestgreen}{rgb}{0.13, 0.55, 0.13}

\algnewcommand{\LineComment}[1]{\Statex \(\triangleright\) #1}

\newcommand{\lambdabf}[0]{\boldsymbol{\lambda}}

\def\convergencePlotWidth{6.5cm}
\def\convergencePlotHeight{5.0cm}

\pgfplotscreateplotcyclelist{convergence_plot}{%
	blue,mark=none\\
	red,mark=none\\
	green,mark=none\\
}
\pgfplotscreateplotcyclelist{convergence_plot_thick_red}{%
	blue,mark=none\\
	red,mark=none,ultra thick\\
	green,mark=none\\
}

\pgfplotscreateplotcyclelist{convergence_k}{%
	black,mark=asterisk\\ 
	blue,mark=*\\ 
	red,mark=square*\\ 
	brown!60!black,mark=triangle*\\ 
}
\pgfplotscreateplotcyclelist{preconditioner_plot}{%
	red,mark=none\\ 
	blue!20,mark=none\\ 
	blue!40,mark=none\\
	blue!60,mark=none\\
	blue!80,mark=none\\
	blue!100,mark=none\\
}
\pgfplotscreateplotcyclelist{heuristics_plot}{%
	blue,mark=square*,mark size=1pt\\
	red,mark=*,mark size=1pt\\
}

\newcommand{\logLogSlopeTriangle}[5]
{

    \pgfplotsextra
    {
        \pgfkeysgetvalue{/pgfplots/xmin}{\xmin}
        \pgfkeysgetvalue{/pgfplots/xmax}{\xmax}
        \pgfkeysgetvalue{/pgfplots/ymin}{\ymin}
        \pgfkeysgetvalue{/pgfplots/ymax}{\ymax}

        \pgfmathsetmacro{\xArel}{#1}
        \pgfmathsetmacro{\yArel}{#3}
        \pgfmathsetmacro{\xBrel}{#1-#2}
        \pgfmathsetmacro{\yBrel}{\yArel}
        \pgfmathsetmacro{\xCrel}{\xArel}

        \pgfmathsetmacro{\lnxB}{\xmin*(1-(#1-#2))+\xmax*(#1-#2)} 
        \pgfmathsetmacro{\lnxA}{\xmin*(1-#1)+\xmax*#1} 
        \pgfmathsetmacro{\lnyA}{\ymin*(1-#3)+\ymax*#3} 
        \pgfmathsetmacro{\lnyC}{\lnyA+#4*(\lnxA-\lnxB)}
        \pgfmathsetmacro{\yCrel}{\lnyC-\ymin)/(\ymax-\ymin)} 

        \coordinate (A) at (rel axis cs:\xArel,\yArel);
        \coordinate (B) at (rel axis cs:\xBrel,\yBrel);
        \coordinate (C) at (rel axis cs:\xCrel,\yCrel);

        \draw[#5]   (A)-- 
                    (B)-- 
                    (C)-- node[pos=0.5,anchor=west] {#4}
                    cycle;
    }
}

\newcommand{\todoNote}[1]%
{\todo[color=yellow!25]{#1}}

\newcommand{\fundingtext}[0]{\textbf{Funding:} This research has been partially funded by the European Union (ERC, NEMESIS, project number 101115663). Views and opinions expressed are however those of the author(s) only and do not necessarily reflect those of the European Union or the European Research Council Executive Agency.
PFA and MV have been partially funded by PRIN2020 n. 20204LN5N5 \emph{``Advanced polyhedral discretisations of heterogeneous PDEs for multiphysics problems''} research grant, funded by the Italian Ministry of Universities and Research (MUR). The present research is part of the activities of ``Dipartimento di Eccelllenza 2023-2027''. PFA and MV are members of INdAM-GNCS.}
 
\title{Iterative solution to the biharmonic equation in mixed form discretized by the Hybrid High-Order method\footnote{\fundingtext}}
\author{P.\ F.\ Antonietti\thanks{MOX, Department of Mathematics, Politecnico di Milano, Italy (\email{paola.antonietti@polimi.it}, \email{pierre.matalon@gmail.com}, \email{marco.verani@polimi.it})}, 
    P.\ Matalon\footnotemark[2]~\footnote{Corresponding author}, 
    M.\ Verani\footnotemark[2]}

\date{}

\begin{document}

\maketitle

\begin{abstract}
We consider the solution to the biharmonic equation in mixed form discretized by the Hybrid High-Order (HHO) methods. 
The two resulting second-order elliptic problems can be decoupled via the introduction of a new unknown, corresponding to the boundary value of the solution of the first Laplacian problem.
This technique yields a global linear problem that can be solved iteratively via a Krylov-type method. More precisely, at each iteration of the scheme, two second-order elliptic problems have to be solved, and a normal derivative on the boundary has to be computed.
In this work, we specialize this scheme for the HHO discretization.
To this aim, an explicit technique to compute the discrete normal derivative of an HHO solution of a Laplacian problem is proposed.
Moreover, we show that the resulting discrete scheme is well-posed.
Finally, a new preconditioner is designed to speed up the convergence of the Krylov method.
Numerical experiments assessing the performance of the proposed iterative algorithm on both two- and three-dimensional test cases are presented.
\end{abstract}


\section{Introduction}

Let $\Omega \subset \mathbb{R}^d$, $d\in \{2, 3\}$, be an open, bounded, polyhedral domain with smooth boundary $\partial\Omega$. In this work, we address the numerical approximation of the solution to the biharmonic equation
\begin{subequations} \label{eq:biharmonic}
\begin{alignat}{2} 
    \Delta^2 \psi = \mathsf{f} \quad &\text{in } \Omega, \label{eq:double_lap} \\
             \psi = \mathsf{g}_\mathrm{D},\; \normalder{\psi} = \mathsf{g}_\mathrm{N} \quad &\text{on } \partial\Omega, \label{eq:bihar_bc}
\end{alignat}
\end{subequations}
where the load function $\mathsf{f} \in L^2(\Omega)$, $\mathsf{g}_\mathrm{D} \in \traceSpace$ and $\mathsf{g}_\mathrm{N} \in \normalderSpace$ are prescribed. As usual, for $X\subset\overline{\Omega}$ and $s\in\mathbb{R}$, we denote by $H^s(X)$ the standard Sobolev space of index $s$.
For all $\mathsf{g}\colon \partial\Omega \to \mathbb{R}$, we define the subspace $H_\mathsf{g}^s(\Omega) \eqbydef \{v \in H^s(\Omega) \mid v_{|\partial\Omega}=\mathsf{g}\}$.
We denote by $(\cdot, \cdot)$ (resp.\ $\langle\cdot, \cdot\rangle$) the standard $L^2$-inner product in $\Omega$ (resp.\ on $\partial\Omega$).
Moreover, $\normalder{\cdot}$ denotes the outer normal derivative on $\partial\Omega$.

Equation \eqref{eq:biharmonic} typically models the bending of a clamped plate supporting a load. 
By introducing the unknown $\omega := -\Delta \psi$, \eqref{eq:double_lap} can be rewritten into two second-order elliptic equations, yielding the mixed formulation
\begin{equation} \label{eq:mixed}
\begin{aligned}
    -\Delta \omega &= \mathsf{f} && \text{in } \Omega, \\
    -\Delta \psi   &= \omega     && \text{in } \Omega, \\
    \psi = \mathsf{g}_\mathrm{D}, \; \normalder{\psi} &= \mathsf{g}_\mathrm{N} && \text{on } \partial\Omega.
\end{aligned}
\end{equation}
This mixed form naturally arises in fluid dynamics, where $\psi$ represents the stream and $\omega$ the vorticity. In the plate bending model, $\psi$ represents the deflection and $\omega$ the bending moment or shear resultant force. 
There are several advantages in employing the mixed formulation \eqref{eq:mixed} over the primal form~\eqref{eq:biharmonic}.
First, while the weak solution of the primal form~\eqref{eq:biharmonic} is to be found in $H^2(\Omega)$, that of the mixed one lies in $H^1(\Omega)$, for which approximation spaces are easier to be constructed. 
Second, the splitting \eqref{eq:mixed} allows, after introducing an additional unknown, the use of fast, scalable solvers available for second-order elliptic equations. 
In particular, we will consider here the Hybrid High-Order (HHO) method \cite{di_pietro_hybrid_2020,di_pietro_lemaire_2014,di_pietro_hybrid_2015}, a non-conforming polyhedral discretization allowing arbitrary polynomial degrees of approximation, and exhibiting optimal convergence rates. 
In this context, the mixed form \eqref{eq:mixed} allows to take advantage of the recent works on multigrid methods designed for the diffusion equation discretized by HHO schemes; see \cite{matalon_h-multigrid_2021,di_pietro_towards_2021,di_pietro_algebraic_2021,di_pietro_high_order_2022}.

Recently, ad-hoc HHO discretizations of the primal formulation~\eqref{eq:biharmonic} have been designed in \cite{dong_hybrid_2022}, and \cite{dong_hybrid_2022_2} generalizes the methods to other boundary conditions. 
See also \cite{dong_hybrid_2021} for HHO discretizations of singularly perturbed fourth-order problems.
Previously, related models have also been tackled with HHO methods, such as the Kirchhoff–Love plate bending model problem in \cite{bonaldi_hybrid_2018}, and the Cahn–Hilliard equation in \cite{chave_hybrid_2016}. 
The mixed form \eqref{eq:mixed} leads to a saddle-point algebraic system, for which special solvers and preconditioners have been proposed, e.g., in \cite{gustafsson_preconditioned_1984,van_gijzen_conjugate_1995,mihajlovic_efficient_2004}.
Nonetheless, in this form, the Laplacian equations are coupled, preventing the use of fast and scalable solvers specifically designed for symmetric positive-definite (SPD) matrices. 
To address this issue, techniques introducing new variables have been designed in order to transform \eqref{eq:mixed} into a series of decoupled problems, whose associated algebraic formulations lead to symmetric and positive-definite algebraic problems.
Among them, we mention the method proposed by Glowinski, Ciarlet, Raviart and Pironneau \cite{ciarlet_mixed_1974,CIARLET1975277,glowinski_numerical_1979}, where the authors introduce the unknown $\lambda \eqbydef \omega_{|\partial\Omega}$. 
In that setting, one solves a sequence of Dirichlet problems and iteratively improves the solution as to enforce the prescribed Neumann condition.
 In a symmetric fashion, the technique proposed by Falk \cite{falk_approximation_1978} introduces $\lambda \eqbydef \normalder{\omega}$, and one solves a sequence of Neumann problems while iteratively improving the solution as to enforce the prescribed Dirichlet condition. 
In both schemes, $\lambda$ is the solution of a linear, symmetric, elliptic equation of the form $\mathcal{L}(\lambda) = b$, in which the evaluation of $\mathcal{L}$ involves the solution of two Laplacian problems. 
In the discrete setting, solving this equation then corresponds to solving an \emph{implicit} linear system, i.e.\ whose matrix is not explicitly known but one can compute its action on a vector. 
This specific configuration is well suited for the use of iterative methods where the operator is applied to a vector at each iteration, without requiring explicit knowledge of the matrix coefficients. Gradient descent algorithms and, more specifically, Krylov methods such as conjugate gradients are ideal candidates in this setting. 

In this work, we focus on the approach of Glowinski \emph{et al.} \cite{ciarlet_mixed_1974,CIARLET1975277,glowinski_numerical_1979}, that we recall in \Cref{sec:mixed_formulation}.
The use of an iterative method for the solution of the global linear problem yields an iterative scheme where each iteration consists of three steps: (i) the solution of a Laplace problem; (ii) the solution of a subsequent Laplacian problem, using the solution of step (i) as a source term; (iii) the computation of the normal derivative on the boundary of the solution at step (ii).
In \Cref{sec:hho_laplace}, we recall the HHO discretization of the Laplacian problem with Dirichlet boundary conditions. 
Then, \Cref{sec:normal_derivative} proposes a computable algorithm for the discrete normal derivative of the HHO solution of a diffusion problem. This step is indeed the main ingredient to define the global discrete scheme.

The work of Glowinski \emph{et al.}\ was applied to the standard Finite Element Method (FEM). 
While that setting does not raise any issue regarding the well-posedness of the discrete problem, as we will show in \Cref{sec:discrete_problem}, in the context of HHO, the resulting problem requires to be stabilized. A stabilization method is proposed and the well-posedness of the discrete problem is proved.

In the context of two-dimensional FEMs, a preconditioner ensuring a convergence rate independent of the mesh size, was proposed in \cite{peisker_numerical_1988}.
It is, however, restricted to two-dimensional problems.
In \Cref{sec:precond}, we therefore propose a novel preconditioner, applicable to general polytopal meshes.
The main idea consists in building an approximate, sparse matrix of the problem, where each column $j$ is computed by solving the two Laplacian problems and evaluating the normal derivative only in a restricted neighbourhood of the $j^{th}$ degree of freedom (DoF).
Finally, numerical experiments are reported in \Cref{sec:numerical_tests}. Various types of two- and three-dimensional meshes, including polygonal meshes, are used.
The scheme exhibits a convergence rate scaling as $\mathcal{O}(h^{k+2})$ in $L^2$-norm, where $k$ denotes the polynomial degree corresponding to the face unknowns of the HHO method.

\section{The continuous splitting of the biharmonic problem} \label{sec:mixed_formulation}

Following \cite{ciarlet_mixed_1974,CIARLET1975277,glowinski_numerical_1979}, we start from equation \eqref{eq:mixed} and we introduce the new unknown $\lambda := \omega_{|\partial\Omega} \in \traceSpace$. Supposing $\lambda$ known, $\omega$ and $\psi$ are successively recovered by solving the following Dirichlet problems:
\begin{subequations} \label{eq:g_subpbfinal}
\begin{equation} \label{eq:g_subpbfinal1}
\left\{
    \begin{alignedat}{2}
    -\Delta \omega &= \mathsf{f} &\quad& \textrm{ in } \Omega, \\
            \omega &= \lambda &\quad& \textrm{ on } \partial \Omega,
    \end{alignedat}
\right.
\end{equation}
\begin{equation} \label{eq:g_subpbfinal2}
\left\{
    \begin{alignedat}{2}
    -\Delta \psi &= \omega &\quad& \textrm{ in } \Omega, \\
            \psi &= \mathsf{g}_{\mathrm{D}} &\quad& \textrm{ on } \partial \Omega.
    \end{alignedat}
\right.
\end{equation}
\end{subequations}
While the Dirichlet datum $\mathsf{g}_{\mathrm{D}}$ is explicitly enforced on the solution in \eqref{eq:g_subpbfinal2}, the enforcement of the Neumann condition $\normalder{\psi} = \mathsf{g}_\mathrm{N}$ defines a problem in the unknown $\lambda$, which is derived in the following way.
For all $\mu \in \traceSpace$, we denote by $(\omega(\mu), \psi(\mu))$ the solution of
\begin{equation*}
\left\{
    \begin{alignedat}{2}
    -\Delta \omega(\mu) &= \mathsf{f} &\quad& \textrm{ in } \Omega, \\
            \omega(\mu) &= \mu &\quad& \textrm{ on } \partial \Omega,
    \end{alignedat}
\right.
\qquad\qquad
\left\{
    \begin{alignedat}{2}
    -\Delta \psi(\mu) &= \omega(\mu) &\quad& \textrm{ in } \Omega, \\
            \psi(\mu) &= \mathsf{g}_{\mathrm{D}} &\quad& \textrm{ on } \partial \Omega.
    \end{alignedat}
\right.
\end{equation*}
By construction, the solution $(\omega, \psi)$ of \eqref{eq:g_subpbfinal} corresponds to $(\omega(\lambda), \psi(\lambda))$. 
Solving \eqref{eq:g_subpbfinal} then boils down to finding $\lambda \in \traceSpace$ such that
\begin{equation} \label{eq:problemtheta}
    \normalder{\psi(\lambda)} = \mathsf{g}_{\mathrm{N}}.
\end{equation}
In order to derive a linear problem from \eqref{eq:problemtheta}, the constraints related to $\mathsf{f}$ and $\mathsf{g}_{\mathrm{D}}$ are eliminated through the introduction of $(\omega_0, \psi_0) \eqbydef (\omega(0), \psi(0))$ (the choice of $\mu=0$ is arbitrary). 
For all $\mu \in \traceSpace$, we then denote by $(\homogC{\omega}{\mu}, \homogC{\psi}{\mu})$ the solution of the following sequence of problems, with vanishing load and Dirichlet function:
\begin{subequations} \label{eq:subpb}
\begin{equation*} \label{eq:subpb1}
\left\{
    \begin{alignedat}{2}
    -\Delta \homogC{\omega}{\mu} &= 0 &\quad& \textrm{ in } \Omega, \\
            \homogC{\omega}{\mu} &= \mu &\quad& \textrm{ on } \partial \Omega,
    \end{alignedat}
\right.
\qquad\qquad
\left\{
    \begin{alignedat}{2}
    -\Delta \homogC{\psi}{\mu} &= \homogC{\omega}{\mu} &\quad& \textrm{ in } \Omega, \\
            \homogC{\psi}{\mu} &= 0 &\quad& \textrm{ on } \partial \Omega,
    \end{alignedat}
\right.
\end{equation*}
\end{subequations}
respectively. 
Equation \eqref{eq:problemtheta} can now be reformulated as the linear problem
\begin{equation} \label{eq:g_lin_equation}
    \mathcal{L}(\lambda) = b, \qquad\qquad b := \normalder{\psi_0}-\mathsf{g}_{\mathrm{N}},
\end{equation}
where $\mathcal{L} \colon \traceSpace \to \normalderSpace$ is the linear operator defined such that for all $\mu \in \traceSpace$, 
\begin{equation} \label{eq:A_G}
    \mathcal{L}(\mu) := -\normalder{\homogC{\psi}{\mu}}.
\end{equation}
The operator $\mathcal{L}$ is proved to be continuous, symmetric and positive-definite in \cite[Lem.\ 2.1]{glowinski_numerical_1979}.

\section{HHO discretization of the Laplacian problem} \label{sec:hho_laplace}

In this section, we briefly recall (see, e.g., \cite[Chap.~2]{di_pietro_hybrid_2020} for extended details) the discrete HHO formulation of the following problem: find $u \colon \Omega \to \mathbb{R}$ such that
\begin{equation} \label{eq:poisson_pb}
\left\{
    \begin{alignedat}{2}
    -\Delta u &= f &\quad& \textrm{ in } \Omega, \\
            u &= g_\mathrm{D} &\quad& \textrm{ on } \partial\Omega,
    \end{alignedat}
\right.
\end{equation}
where $f \colon \Omega \to \mathbb{R}$ and $g_\mathrm{D} \in \traceSpace$.
A weak solution of \eqref{eq:poisson_pb} is obtained via the variational formulation: 
find $u\in H^1_{g_\mathrm{D}}(\Omega)$ such that
\begin{equation} \label{eq:laplace_weak}
    a(u, v) = (f, v) \qquad \forall v\in H_0^1(\Omega),
\end{equation}
where the bilinear form $a$ is such that $a(v, w) = (\nabla v, \nabla w)$, for all $v, w\in H^1(\Omega)$.

\subsection{Mesh definition and notation} \label{sec:mesh_definition}

Let the couple $(\cells, \faces)$ define a mesh of the domain $\Omega \subset \mathbb{R}^d$, $d\in\{1,2,3\}$: $\cells$ is a set of disjoint, open, polyhedral elements such that $\bigcup_{T\in \cells}\overline{T} = \overline{\Omega}$; $\faces$ is the set of element faces; $h \eqbydef \max_{T\in \cells} h_T$ with $h_T$ denoting the diameter of~$T\in\cells$. 
The mesh is assumed to match the geometrical requirements of \cite[Def.~1.4]{di_pietro_hybrid_2020} and, when asymptotic behaviours are studied, polytopal regular mesh sequences in the sense defined by \cite[Def.~1.9]{di_pietro_hybrid_2020} are considered.
Let $\bcells$ the subset of $\cells$ collecting the elements located at the boundary of the domain.
We also define the following subsets of $\faces$: 
\begin{itemize}
    \item $\ifaces$, collecting the interior faces;
    \item $\bfaces$, collecting the boundary faces;
    \item $\mathcal{F}_T$, collecting the faces of $T$, for all $T\in \cells$.
\end{itemize}
We denote by $\mathbf{n}_{\partial T}$ the unit normal vector to $\partial T$ pointing outward of $T$.
For all $T \in \cells$ (resp.\ $F\in \faces$), we denote by $(\cdot, \cdot)_T$ (resp.\ $\langle \cdot, \cdot\rangle_F$) the standard inner product of $L^2(T)$ (resp.\ $L^2(F)$) or $L^2(T)^d$. 
We also denote by $\langle \cdot, \cdot\rangle_{\facesT} \eqbydef \sum_{F\in\facesT} \langle \cdot, \cdot\rangle_F$

\subsection{Local and broken polynomial spaces}

The HHO method hinges on discrete unknowns representing polynomial functions local to elements and faces. 
So, for all $m \in \mathbb{N}_0$ and all $T\in \cells$ (resp.\ $F\in \faces$), we denote by $\mathbb{P}^m(T)$ (resp.\ $\mathbb{P}^m(F)$) the space spanned by the restriction to $T$ (resp.\ $F$) of $d$-variate polynomials of total degree $\leq m$. From these local polynomial spaces, we can construct the following broken polynomial spaces supported by the mesh and its skeleton:
\begin{align*}
    \brokenPolyT{h}{m} &\eqbydef \left\lbrace v_{\cells} \eqbydef (v_T)_{T\in\cells} \mid v_T \in \mathbb{P}^{m}(T) \;\; \forall T \in \cells \right\rbrace , \\
    \mathbb{P}^{m}(\mathcal{F}_h) &\eqbydef \left\lbrace v_{\mathcal{F}_h} \eqbydef (v_F)_{F\in\mathcal{F}_h} \mid v_F \in \mathbb{P}^m(F) \;\; \forall F \in \mathcal{F}_h \right\rbrace,
\end{align*}
respectively. 
The local space $\mathbb{P}^{m}(\mathcal{F}_T)$ is defined analogously for all $T\in\cells$.
For all cell or face $X$, we denote by $\pi_X^m \colon L^2(X) \to \mathbb{P}^m(X)$ the local $L^2$-orthogonal projector onto the space $\mathbb{P}^m(X)$. 
By patching up those local projectors, we denote by $\pi_{\cells}^m \colon L^2(\Omega) \to \mathbb{P}^m(\cells)$ the piecewise $L^2$-orthogonal projector onto $\mathbb{P}^m(\cells)$, and by $\pi_{\bfaces}^m \colon L^2(\partial\Omega) \to \mathbb{P}^m(\bfaces)$ the piecewise $L^2$-orthogonal projector onto $\mathbb{P}^m(\bfaces)$.

\subsection{Discrete hybrid formulation}

Given the polynomial degrees $k \in \mathbb{N}_0$ and $l \in \{k, k+1\}$.
The global and local spaces of \emph{hybrid} variables are defined as
\begin{align*} \label{eq:local_hybrid_space}
\hybridSpace &\eqbydef \left\lbrace \underline{v}_h \eqbydef \left( v_{\cells}, v_{\faces} \right) \in \brokenPolyT{h}{l} \times \brokenPolyF{h}{k} \right\rbrace, \\
\locHybridSpace &\eqbydef \left\lbrace \underline{v}_T \eqbydef \left( v_T, v_{\facesT} \right) \in \mathbb{P}^l(T)\times \brokenPolyF{T}{k} \right\rbrace \quad \forall T\in\cells,
\end{align*}
respectively. 
For any $\underline{v}_h \in \hybridSpace$, we denote by $\underline{v}_T \in \locHybridSpace$ its restriction to $T \in \cells$.
Boundary data are strongly accounted for in the following subspaces:
\begin{equation*}
    \brokenPolyF{h}{k,g_\mathrm{D}} \eqbydef \left\lbrace v_{\faces} \in \brokenPolyF{h}{k} \mid v_F = \pi_F^k\, g_\mathrm{D} \;\; \forall F \in \bfaces \right\rbrace,
    \qquad
    \hybridSpaceDirich{g_\mathrm{D}} \eqbydef \brokenPolyT{h}{l} \times \brokenPolyF{h}{k,g_\mathrm{D}}.
\end{equation*} 
In particular, homogeneous Dirichlet conditions are strongly enforced in $\hybridSpaceDirich{0}$.
The global HHO bilinear form associated to the variational formulation of problem \eqref{eq:poisson_pb} is defined as $a_h \colon \hybridSpace \times \hybridSpace \to \mathbb{R}$ such that
$a_h(\underline{v}_h, \underline{w}_h) \eqbydef \sum_{T \in \cells} a_T(\underline{v}_T, \underline{w}_T)$
where the local bilinear form $a_T \colon \locHybridSpace \times \locHybridSpace \to \mathbb{R}$ is defined as
\begin{equation} \label{eq:loc_bilin_form}
a_T(\underline{v}_T, \underline{w}_T) \eqbydef (\nabla p_T^{k+1}\underline{v}_T, \nabla p_T^{k+1}\underline{w}_T)_T + s_T(\underline{v}_T, \underline{w}_T) \qquad \forall T \in \cells, 
\end{equation}
In this expression, the first term is responsible for consistency while the second is required to ensure stability of the scheme. 
The consistency term involves the \emph{local potential reconstruction operator} $p_T^{k+1} \colon \locHybridSpace \to \mathbb{P}^{k+1}(T)$ defined such that, for all $\underline{v}_T \in \locHybridSpace$, it satisfies
\begin{subequations} \label{eq:reconstruct_operator}
\begin{empheq}[left = \empheqlbrace]{align}
&(\nabla p_T^{k+1}\underline{v}_T, \nabla w)_T = -(v_T, \Delta w)_T + \langle v_{\facesT}, \nabla w \cdot \mathbf{n}_{\partial T} \rangle_{\facesT} \qquad\quad \forall w \in \mathbb{P}^{k+1}(T), \label{eq:reconstruct_operator_gradient} \\
&(p_T^{k+1}\underline{v}_T, 1)_T = (v_T, 1)_T. \label{eq:reconstruct_operator_constant_fixing}
\end{empheq}
\end{subequations}
Given the local interpolant $\underline{v}_T \in \locHybridSpace$ of a function $v\in L^2(\Omega)$, $p_T^{k+1}$ reconstructs an approximation of $v$ of degree $k+1$.
By patching together local contributions, we define the global operator $p_h^{k+1}\colon \hybridSpace \to \mathbb{P}^{k+1}(\cells)$ such that $(p_h^{k+1} \underline{v}_h)_{|T} \eqbydef p_T^{k+1} \underline{v}_T$ for all $T\in\cells$.
The stabilization bilinear form $s_T$ depends on its arguments through the \emph{difference operators} $\delta_T:\locHybridSpace\to\mathbb{P}^k(T)$ and $\delta_{TF}:\locHybridSpace\to\mathbb{P}^k(F)$ for all $F\in\mathcal{F}_T$, defined such that, for all $\underline{v}_T\in\locHybridSpace$,
\begin{equation} \label{eq:diff_operators}
\text{
$\delta_T \underline{v}_T \eqbydef \pi_T^l(p_T^{k+1}\underline{v}_T - v_T)$ and
$\delta_{TF} \underline{v}_T \eqbydef \pi_F^k(p_T^{k+1}\underline{v}_T - v_F)$ for all $F \in \mathcal{F}_T$.
}
\end{equation}
These operators capture the higher-order correction that the reconstruction $p_T^{k+1}$ adds to the element and face unknowns, respectively.
A classical expression for $s_T \colon \locHybridSpace \times \locHybridSpace \to \mathbb{R}$ is
\begin{equation} \label{eq:stabilization}
    s_T(\underline{v}_T, \underline{w}_T) \eqbydef \sum_{F \in \mathcal{F}_T} h_F^{-1} \; \mathfrak{s}_{TF} (\underline{v}_T, \underline{w}_T), 
    \qquad\qquad
    \mathfrak{s}_{TF}(\underline{v}_T, \underline{w}_T) \eqbydef \langle(\delta_{TF} - \delta_T) \underline{v}_T, (\delta_{TF} - \delta_T) \underline{w}_T \rangle_F.
\end{equation}
If $l=k+1$, we can also use the simpler formula (used, e.g., in \cite{chave_hybrid_2016})
\begin{equation} \label{eq:stab2}
    \mathfrak{s}_{TF}(\underline{v}_T, \underline{w}_T) \eqbydef \langle \pi_F^k(v_T - v_F), \pi_F^k( w_T - w_F) \rangle_F. 
\end{equation}
The global HHO problem reads: find $\underline{u}_h \in \hybridSpaceDirich{g_\mathrm{D}}$ such that
\begin{equation} \label{eq:hho_formulation}
a_h(\underline{u}_h, \underline{v}_h) = (f, v_{\cells}) \qquad \forall  \underline{v}_h \in \hybridSpaceDirich{0}.
\end{equation}
The final approximation is obtained through the post-processing step
\begin{equation} \label{eq:higher_order_reconstruction}
    u_h \eqbydef p_h^{k+1}\underline{u}_h \in \brokenPolyT{h}{k+1}.
\end{equation}

\subsection{Local problems and cell unknown recovery operator}

It follows from this local construction that the cell unknowns are only locally coupled. 
For all $T\in\cells$, we define the linear operator $\rcvT \colon \brokenPolyF{T}{k} \to \mathbb{P}^{l}(T)$ such that for all $v_{\facesT} \in \brokenPolyF{T}{k}$, $\rcvT v_{\facesT}$ is the unique solution of the local problem
\begin{equation} \label{eq:local_subproblem_1}
    a_T((\rcvT v_{\facesT}, 0), (w_T, 0)) = - a_T((0, v_{\facesT}), (w_T, 0)) \qquad \forall w_T \in \mathbb{P}^{l}(T).
\end{equation}
We denote by $\rcvT$ the \emph{cell unknown recovery operator}.
In order to shorten the notation of the hybrid couple $\underline{v}_T := (\rcvT v_{\facesT}, v_{\facesT})$, we define the operator $\underline{\Theta}_T \colon \brokenPolyF{T}{k} \to \locHybridSpace$ such that for all $v_{\facesT}$,
\begin{equation} \label{eq:cell_rcv_hybrid}
\underline{\Theta}_T v_{\facesT} := (\rcvT v_{\facesT}, v_{\facesT}).
\end{equation}
Finally, the associated global operators $\rcvtf \colon \brokenPolyF{h}{k} \to \brokenPolyT{h}{l}$ and $ \underline{\Theta}_h \colon \brokenPolyF{h}{k} \to \hybridSpace$ are defined locally such that for all $T \in \cells$:
\begin{align}
(\rcvtf v_{\faces})_{|T} \eqbydef \rcvT v_{\facesT} \qquad \forall v_{\faces} \in \brokenPolyF{h}{k} \label{eq:cell_rcv_global}, \\
(\underline{\Theta}_h v_{\faces})_{|T\times\partial T} \eqbydef \underline{\Theta}_T v_{\facesT} \qquad \forall v_{\faces} \in \brokenPolyF{h}{k}. \nonumber
\end{align}
$\underline{\Theta}_h$ verifies the following useful property:

\begin{lemma} \label{lem:link_condensed_hybrid}
For all $v_{\faces} \in \brokenPolyF{h}{k}$ and $\underline{w}_h \in \hybridSpace$, it holds that
\begin{align} 
a_h(\underline{\Theta}_h v_{\faces}, \underline{w}_h) &= a_h(\underline{\Theta}_h v_{\faces}, \underline{\Theta}_h w_{\faces}). \label{eq:link_condensed_hybrid}
\end{align} 
\end{lemma}
\begin{proof}
    See \Cref{annex}.
\end{proof}



\begin{remark} 
    (Static condensation) 
    In practice, problem \eqref{eq:hho_formulation} is solved through the equivalent, condensed formulation:
find $u_{\faces} \in \brokenPolyF{h}{k,g_\mathrm{D}}$ such that
\begin{equation} \label{eq:static_condensed_problem}
\widehat{a}_h(u_{\faces}, v_{\faces}) 
= 
(f, \rcvtf v_{\faces}) \qquad \forall v_{\faces} \in \brokenPolyF{h}{k,0},
\end{equation}
where $\widehat{a}_h \colon \brokenPolyF{h}{k} \times \brokenPolyF{h}{k} \to \mathbb{R}$ is such that for all $v_{\faces}, w_{\faces} \in \brokenPolyF{h}{k}$,
\begin{equation} \label{eq:static_condensed_bilin_form}
\widehat{a}_h(v_{\faces}, w_{\faces}) := a_h(\underline{\Theta}_h v_{\faces}, \underline{\Theta}_h w_{\faces}).
\end{equation}
Refer to \cite[Prop.\ 4]{cockburn_bridging_2016} for the proof.
\end{remark}
\section{Discrete normal derivative} \label{sec:normal_derivative}

In this section, we derive an approximation, in the HHO context, of the normal derivative on $\partial \Omega$. This formula will play a crucial role in the approximation scheme of the biharmonic problem. 



Let $u$ be the solution of the boundary value problem \eqref{eq:poisson_pb} and 
let $\mathcal{H} \colon \traceSpace \to H^1(\Omega)$ be a linear operator such that for all $v \in \traceSpace$ defined on the boundary, $\mathcal{H}v$ extends $v$ in the interior of $\Omega$. 
As the trace operator is surjective, such an operator exists.
By Green's formula, it holds that
\begin{equation*}
    \langle \normalder{u}, v \rangle = (\nabla u, \nabla \mathcal{H}v) + (\Delta u, \mathcal{H}v) \qquad \forall v \in \traceSpace,
\end{equation*}
where the $L^2(\partial\Omega)$-inner product notation $\langle \cdot, \cdot \rangle$ is employed here to denote a duality pairing.
Using the bilinear form $a(\cdot,\cdot)$, and given that $-\Delta u = f$, the above equation becomes
\begin{equation} \label{eq:normalder_ipp}
    \langle \normalder{u}, v \rangle = a(u, \mathcal{H}v) - (f, \mathcal{H}v) \qquad \forall v \in \traceSpace.
\end{equation}
In the literature, the operator $g_{\mathrm{D}} \mapsto \normalder{u}$ is called Poincaré-Steklov operator, or Dirichlet-to-Neumann map; see, e.g. \cite{spigler_theory_1991}.
Equation \eqref{eq:normalder_ipp} is the well-known variational formula for the computation of the normal derivative.
Now, given the solution $\underline{u}_h$ of the discrete problem \eqref{eq:hho_formulation}, we define its normal derivative on the boundary faces, denoted by $\flux(\underline{u}_h)\in \mathbb{P}^k(\bfaces)$, through a discrete counterpart of \eqref{eq:normalder_ipp} in the HHO setting. 
Namely, $\flux(\underline{u}_h)$ verifies
\begin{equation} \label{eq:flux_computation}
    \langle \flux(\underline{u}_h), v_{\bfaces} \rangle 
    = 
    a_h(\underline{u}_h, \underline{\mathcal{H}}_h v_{\bfaces}) - (f, \mathcal{H}_{\cells}v_{\bfaces}) \qquad \forall v_{\bfaces} \in \mathbb{P}^k(\bfaces),
\end{equation}
where $\underline{\mathcal{H}}_h$ is an extension/lifting operator expressed in the hybrid setting, i.e. 
\begin{equation*}
    \underline{\mathcal{H}}_h \eqbydef (\mathcal{H}_{\cells}, \mathcal{H}_{\faces}),\qquad
    \mathcal{H}_{\cells} \colon \mathbb{P}^k(\bfaces) \to \brokenPolyT{h}{l},\qquad
    \mathcal{H}_{\faces} \colon \mathbb{P}^k(\bfaces) \to \brokenPolyF{h}{k}.
\end{equation*}
For all $v_{\bfaces} \in \mathbb{P}^k(\bfaces)$ and $F\in\faces$, we define
\begin{equation*} 
    (\mathcal{H}_{\faces} v_{\bfaces})_{|F} \eqbydef
    \begin{cases}
        (v_{\bfaces})_{|F} & \text{ if } F \in \bfaces, \\
        0 & \text{ otherwise}.
    \end{cases}
\end{equation*}
Then, we set
\begin{equation*}
\mathcal{H}_{\cells} \eqbydef \rcvtf \mathcal{H}_{\faces},
\end{equation*}
where $\rcvtf$ is defined as in \eqref{eq:cell_rcv_global}.


\begin{remark}
(Condensed formula)
In order to facilitate the practical computation of $\flux(\underline{u}_h)$, equation \eqref{eq:flux_computation} can be rewritten in terms of the \emph{condensed} bilinear form $\widehat{a}_h$ (cf.\ \eqref{eq:static_condensed_bilin_form}). 
First of all, it follows from the definition of $\underline{\Theta}_h$ that 
$\underline{\mathcal{H}}_h = \underline{\Theta}_h \mathcal{H}_{\faces}$.
Then, by using the symmetry of $a_h$ with property \eqref{eq:link_condensed_hybrid} and the definition \eqref{eq:static_condensed_bilin_form} of $\widehat{a}_h$, we have
\begin{equation*}
    a_h(\underline{u}_h, \underline{\mathcal{H}}_h v_{\bfaces})
    =
    a_h(\underline{u}_h, \underline{\Theta}_h \mathcal{H}_{\faces}v_{\bfaces})
    =
    a_h(\underline{\Theta}_h u_{\faces}, \underline{\Theta}_h \mathcal{H}_{\faces}v_{\bfaces})
    =
    \widehat{a}_h(u_{\faces}, \mathcal{H}_{\faces}v_{\bfaces}).
\end{equation*}
Equation \eqref{eq:flux_computation} then becomes
\begin{equation} \label{eq:flux_computation_condensed}
    \langle \flux(\underline{u}_h), v_{\bfaces} \rangle 
    = 
    \widehat{a}_h(u_{\faces}, \mathcal{H}_{\faces}v_{\bfaces}) - (f, \mathcal{H}_{\cells} v_{\bfaces}) \qquad 
    \forall v_{\bfaces} \in \mathbb{P}^k(\bfaces).
\end{equation}
Notice that compared to \eqref{eq:flux_computation}, which requires the knowledge of $\underline{u}_h$, formula \eqref{eq:flux_computation_condensed} only involves $u_{\faces}$. Consequently, the following hold: (i) $u_{\cells}$ does not have to be computed, (ii) the matrix used to compute the first term is smaller, as the matrix representation of $\widehat{a}_h$ is a Schur complement where the cell unknowns have been eliminated. 
\end{remark}

The computation of the discrete normal derivative as proposed here is numerically validated in \Cref{sec:num_tests_normalder}. The experiments show a convergence in $\mathcal{O}(h^{k+1})$ in $L^2$-norm.

\begin{remark}
    (Choice of discrete normal derivative)
    In this scheme, the normal derivative on the boundary is defined by a discrete version of \eqref{eq:normalder_ipp}, i.e.\ through an integration by parts involving a lifting operator, which is the way employed by Glowinski \emph{et al.} in their original scheme \cite{glowinski_numerical_1979} with standard FEM.
    However, contrary to FEM, HHO methods natively include a flux formulation (cf.\ \cite[Lem.\ 2.25]{di_pietro_hybrid_2020}), which we could lean on instead of \eqref{eq:normalder_ipp}.
    However, the structure of formula \eqref{eq:normalder_ipp} directly interplays the two Laplacian problems in variational form, allowing a natural reformulation of the problem through a symmetric positive-definite bilinear form (see \cref{lem:lh_reformulation}).
\end{remark}

\section{Discrete HHO problem} \label{sec:discrete_problem}

In this section, the global, discrete realization of problem \eqref{eq:g_lin_equation} is formulated.

\subsection{Bilinear form}

In order to devise a stable problem, we will make use of the following bilinear form in the hybrid space $\hybridSpace$. For all $\underline{v}_h, \underline{w}_h \in\hybridSpace$, we introduce
\begin{equation} \label{eq:L2_like_innerprod}
    \stabcellIP{\underline{v}_h, \underline{w}_h} 
    \eqbydef 
    (v_{\cells}, w_{\cells}) 
    + \sum_{T \in \bcells} \sum_{F \in \mathcal{F}_T} h_F \; \mathfrak{s}_{TF}^\star(\underline{v}_T, \underline{w}_T),
    \qquad\quad
    \mathfrak{s}_{TF}^\star(\underline{v}_T, \underline{w}_T) \eqbydef \langle \pi_F^k(v_T - v_F), \pi_F^k( w_T - w_F) \rangle_F.
\end{equation}
Formula \eqref{eq:L2_like_innerprod} defines an inner product-like bilinear form based on the $L^2$-inner product of the cell unknowns, to which a stabilizing term has been added.
The stabilizing term is inspired from \cite[Eq.\ (26)]{chave_hybrid_2016}.
Remark that $\mathfrak{s}_{TF}^\star$ is the same as \eqref{eq:stab2}, that one can employ to stabilize the Laplacian bilinear form if $l=k+1$.
The scaling factor $h_F$ is selected to ensure dimensional homogeneity with the consistency term it is added to.
Note that $\stabcellIP{\cdot, \cdot}$ does not define an inner product in $\hybridSpace$, due to the fact that only the boundary cells are involved in the stabilizing term.
In spite of that, we allow ourselves the use of an inner product notation, since $\stabcellIP{\cdot, \cdot}$ shall undertake the role of the $L^2$-inner product in the discrete problem and, thus, carries the same semantics.

$\stabcellIP{\cdot, \cdot}$ remains symmetric and positive semi-definite.
If definiteness is not globally ensured, it is ensured ``locally'' in the cells where the stabilization term is applied. One might say that $\stabcellIP{\cdot, \cdot}$ enjoys a property of \emph{local stability} in the boundary cells, which is formalized by
\begin{lemma} (Local stability on the boundary)
\begin{equation} \label{eq:local_stability}
    \stabcellIP{\underline{v}_h, \underline{v}_h} = 0 
    \qquad \Longrightarrow \qquad
    \underline{v}_T = 0 \quad \forall T\in\bcells.
\end{equation}
\end{lemma}
\begin{proof}
    Let $T\in\bcells$.
    Assume $\underline{v}_h \in \hybridSpace$ such that $\stabcellIP{\underline{v}_h, \underline{v}_h} = 0$.
    Then, by definition \eqref{eq:L2_like_innerprod},
    \begin{equation*}
        (v_T, v_T)_T  + \sum_{F \in \mathcal{F}_T} h_F \; \mathfrak{s}_{TF}^\star(\underline{v}_T, \underline{v}_T) = 0.
    \end{equation*}
    As $\mathfrak{s}_{TF}^\star$ is positive semi-definite, all terms are non-negative. 
    Consequently, $v_T = 0$ and $\mathfrak{s}_{TF}^\star(\underline{v}_T, \underline{v}_T) = 0$ for all $F\in\mathcal{F}_T$.
    We then have
    \begin{equation*}
        0 = \mathfrak{s}_{TF}^\star(\underline{v}_T, \underline{v}_T) = \| \pi_F^k(v_T - v_F) \|_F^2 = \| v_F \|_F^2,
    \end{equation*}
    where we have used the facts that $v_T = 0$ and $\pi_F^k v_F = v_F$.
    We conclude that $v_F = 0$.
\end{proof}


Let us now describe the discrete, variational form of the continuous operator \eqref{eq:A_G}.
For all $\mu \in \mathbb{P}^k(\bfaces)$, we denote by $(\homog{\underline{\omega}}{h}{\mu}, \homog{\underline{\psi}}{h}{\mu}) \in \hybridSpaceDirich{\mu} \times \hybridSpaceDirich{0}$ the solution of the discrete problems
\begin{subequations} \label{eq:discrete_mixed_problem}
\begin{align}
    a_h(\homog{\underline{\omega}}{h}{\mu}, \underline{v}_h) &= 0 &\forall \underline{v}_h \in \hybridSpaceDirich{0}, \label{eq:discrete_subproblem1}\\
    a_h(\homog{\underline{\psi}}{h}{\mu}, \underline{v}_h) &= \stabcellIP{\homog{\underline{\omega}}{h}{\mu}, \underline{v}_h} 
    &\forall \underline{v}_h \in \hybridSpaceDirich{0}.\label{eq:discrete_subproblem2}
\end{align}
\end{subequations}
For $\mu, \eta\in \mathbb{P}^k(\bfaces)$, and $(\homog{\underline{\omega}}{h}{\mu}, \homog{\underline{\psi}}{h}{\mu})$ the solution of \eqref{eq:discrete_mixed_problem} associated with $\mu$, we define the bilinear form $\ell_h \colon \mathbb{P}^k(\bfaces)\times \mathbb{P}^k(\bfaces) \to \mathbb{R}$ as the computation of $-\flux(\homog{\underline{\psi}}{h}{\mu})$ via formula \eqref{eq:flux_computation}, such that
\begin{equation} \label{eq:lh_definition}
    \ell_h (\mu, \eta) \eqbydef - a_h(\homog{\underline{\psi}}{h}{\mu}, \underline{\mathcal{H}}_h\eta)
    +
    \stabcellIP{\homog{\underline{\omega}}{h}{\mu}, \underline{\mathcal{H}}_h\eta}. 
\end{equation}
Remark that we have used $\stabcellIP{\cdot, \cdot}$ in \eqref{eq:discrete_subproblem2} and \eqref{eq:lh_definition} instead of $(\cdot, \cdot)$ in the reference formulas \eqref{eq:hho_formulation} and \eqref{eq:flux_computation}, respectively.
Given that formula \eqref{eq:lh_definition} does not explicitly exhibit symmetry or positive-definiteness, we prove an equivalent reformulation, easier to analyze.
\begin{lemma} \label{lem:lh_reformulation} (Reformulation of $\ell_h$)
For $\mu$ and $\eta \in \mathbb{P}^k(\bfaces)$, denote by $\homog{\underline{\omega}}{h}{\mu}$ and $\homog{\underline{\omega}}{h}{\eta}$ their associated solutions of \eqref{eq:discrete_subproblem1}, respectively.
Then 
\begin{equation} \label{eq:lh_reformulation}
    \ell_h(\mu, \eta) 
    =
    \stabcellIP{\homog{\underline{\omega}}{h}{\mu}, \homog{\underline{\omega}}{h}{\eta}}.
\end{equation}
\end{lemma}

\begin{proof}
For $\mu$ and $\eta \in \mathbb{P}^k(\bfaces)$, let $(\homog{\underline{\omega}}{h}{\mu}, \homog{\underline{\psi}}{h}{\mu})$ and $(\homog{\underline{\omega}}{h}{\eta}, \homog{\underline{\psi}}{h}{\eta})$ be their respective, associated solutions of \eqref{eq:discrete_mixed_problem}.
Let us start with the definition \eqref{eq:lh_definition} of $\ell_h$:
\begin{equation} \label{eq:lh:proof:1}
    \ell_h(\mu, \eta) 
    \eqbydef
    - a_h(\homog{\underline{\psi}}{h}{\mu}, \underline{\mathcal{H}}_h\eta)
    +
    \stabcellIP{\homog{\underline{\omega}}{h}{\mu}, \underline{\mathcal{H}}_h\eta}.
\end{equation}
By writing $\underline{\mathcal{H}}_h\eta = (\underline{\mathcal{H}}_h\eta - \homog{\underline{\omega}}{h}{\eta}) + \homog{\underline{\omega}}{h}{\eta}$, the first term becomes
\begin{equation} \label{eq:lh:proof:2}
    a_h(\homog{\underline{\psi}}{h}{\mu}, \underline{\mathcal{H}}_h\eta)
    =
    a_h(\homog{\underline{\psi}}{h}{\mu}, \underline{\mathcal{H}}_h\eta - \homog{\underline{\omega}}{h}{\eta})
    +
    a_h(\homog{\underline{\psi}}{h}{\mu}, \homog{\underline{\omega}}{h}{\eta}).
\end{equation}
As $\underline{\mathcal{H}}_h$ does not modify the boundary face unknowns, we have $\underline{\mathcal{H}}_h\eta  \in \hybridSpaceDirich{\eta}$. 
Additionally, $\homog{\underline{\omega}}{h}{\eta} \in \hybridSpaceDirich{\eta}$ by definition, so the difference $(\underline{\mathcal{H}}_h\eta - \homog{\underline{\omega}}{h}{\eta}) \in \hybridSpaceDirich{0}$. Consequently, as $\homog{\underline{\psi}}{h}{\mu}$ verifies \eqref{eq:discrete_subproblem2}, it holds that
\begin{equation} \label{eq:lh:proof:3}
    a_h(\homog{\underline{\psi}}{h}{\mu}, \underline{\mathcal{H}}_h\eta - \homog{\underline{\omega}}{h}{\eta}) 
    = 
    \stabcellIP{\homog{\underline{\omega}}{h}{\mu}, \underline{\mathcal{H}}_h\eta - \homog{\underline{\omega}}{h}{\eta}} 
    = \stabcellIP{\homog{\underline{\omega}}{h}{\mu}, \underline{\mathcal{H}}_h\eta} 
    - \stabcellIP{\homog{\underline{\omega}}{h}{\mu}, \homog{\underline{\omega}}{h}{\eta}}.
\end{equation}
For the second term of \eqref{eq:lh:proof:2}, given that $\homog{\underline{\omega}}{h}{\eta}$ is solution of \eqref{eq:discrete_subproblem1} and $\homog{\underline{\psi}}{h}{\mu} \in \hybridSpaceDirich{0}$, we can use $\homog{\underline{\psi}}{h}{\mu}$ as a test function in \eqref{eq:discrete_subproblem1} to show, since $a_h$ is symmetric, that
\begin{equation} \label{eq:lh:proof:4}
    a_h(\homog{\underline{\psi}}{h}{\mu}, \homog{\underline{\omega}}{h}{\eta}) = 0.
\end{equation}
Now, plugging \eqref{eq:lh:proof:3} and \eqref{eq:lh:proof:4} into \eqref{eq:lh:proof:2} gives
\begin{equation} \label{eq:lh:proof:5}
    a_h(\homog{\underline{\psi}}{h}{\mu}, \underline{\mathcal{H}}_h\eta)
    =
    \stabcellIP{\homog{\underline{\omega}}{h}{\mu}, \underline{\mathcal{H}}_h\eta} 
    - \stabcellIP{\homog{\underline{\omega}}{h}{\mu}, \homog{\underline{\omega}}{h}{\eta}} 
    .
\end{equation}
Finally, plugging \eqref{eq:lh:proof:5} into \eqref{eq:lh:proof:1} yields \eqref{eq:lh_reformulation}.
\end{proof}

Using formulation \eqref{eq:lh_reformulation}, we clearly have
\begin{theorem} \label{th:spd}
$\ell_h$, defined by \eqref{eq:lh_definition}, is symmetric and positive-definite.
\end{theorem}
\begin{proof}
    Formulation \eqref{eq:lh_reformulation} clearly shows symmetry.
    Additionally, for all $\mu\in\mathbb{P}^k(\bfaces)$, we have
\begin{equation} \label{eq:lh:proof:7}
    \ell_h(\mu, \mu) 
    = \stabcellIP{\homog{\underline{\omega}}{h}{\mu}, \homog{\underline{\omega}}{h}{\mu}}.
\end{equation}
Since $\stabcellIP{\cdot, \cdot}$ is positive semi-definite, so is $\ell_h$. 
Now, assume $\ell_h(\mu, \mu) = 0$, which, by \eqref{eq:lh:proof:7}, means $\stabcellIP{\homog{\underline{\omega}}{h}{\mu}, \homog{\underline{\omega}}{h}{\mu}} = 0$.
The property of local stability \eqref{eq:local_stability} implies that, for all $T\in\bcells$, we have $\homog{\underline{\omega}}{T}{\mu} = 0$, i.e.\ $\homog{\omega}{T}{\mu} = 0$ and $\homog{\omega}{\facesT}{\mu} = 0$.
In particular, for all $F\in\bfaces$, $\mu_F = \homog{\omega}{F}{\mu} = 0$.
This being true for all $F\in\bfaces$, we have $\mu=0$, which proves that $\ell_h$ is positive-definite.
\end{proof}


\subsection{Discrete scheme}

Let $(\underline{\omega}_h^0, \underline{\psi}_h^0) \in \hybridSpaceDirich{0} \times \hybridSpaceDirich{\mathsf{g}_{\mathrm{D}}}$ be the solution of
\begin{subequations} \label{eq:disc_subpb_before}
\begin{align} 
    a_h(\underline{\omega}_h^0, \underline{v}_h) &= (\mathsf{f}, v_{\cells}) &\forall \underline{v}_h \in \hybridSpaceDirich{0},\\
    a_h(\underline{\psi}_h^0, \underline{v}_h) &= (\omega_{\cells}^0, v_{\cells}) &\forall \underline{v}_h \in \hybridSpaceDirich{0}.
\end{align}
\end{subequations}
The discrete formulation of problem \eqref{eq:g_lin_equation} reads: find $\lambda_{\bfaces} \in \mathbb{P}^k(\bfaces)$ such that
\begin{equation} \label{eq:discrete_pb}
    \ell_h (\lambda_{\bfaces}, \mu) = \langle \flux(\underline{\psi}_h^0) - \mathsf{g}_{\mathrm{N}}, \;\mu \rangle
    \qquad
    \forall \mu \in \mathbb{P}^k(\bfaces).
\end{equation}
We also define the linear operator $\mathcal{L}_h \colon \mathbb{P}^k(\bfaces) \to \mathbb{P}^k(\bfaces)$ associated to the bilinear form $\ell_h$ such that, for all $\mu \in \mathbb{P}^k(\bfaces)$,
\begin{equation} \label{eq:linear_operator}
    \langle \mathcal{L}_h\mu, \eta \rangle \eqbydef \ell_h(\mu, \eta) \qquad \forall \eta \in \mathbb{P}^k(\bfaces).
\end{equation}
Since $\ell_h$ is positive-definite, problem \eqref{eq:discrete_pb} is well-posed. Its solution exists and is unique. Once it is solved, we compute $(\underline{\omega}_h, \underline{\psi}_h) \in \hybridSpaceDirich{\lambda_{\bfaces}} \times \hybridSpaceDirich{\mathsf{g}_{\mathrm{D}}}$, solution of
\begin{subequations} \label{eq:disc_subpb_after}
\begin{align} 
    a_h(\underline{\omega}_h, \underline{v}_h) &= (\mathsf{f}, v_{\cells}) &\forall \underline{v}_h \in \hybridSpaceDirich{0},\\
    a_h(\underline{\psi}_h, \underline{v}_h) &= (\omega_{\cells}, v_{\cells}) &\forall \underline{v}_h \in \hybridSpaceDirich{0}.
\end{align}
\end{subequations}
Finally, the discrete approximation of $(\omega, \psi)$ is given by
\begin{equation*}
    (\omega_h, \psi_h) \eqbydef (p_h^{k+1}\underline{\omega}_h, p_h^{k+1}\underline{\psi}_h).
\end{equation*}
Note that while $\stabcellIP{\cdot, \cdot}$ is required within $\ell_h$ for problem \eqref{eq:discrete_pb} to be well-posed, it has to be used neither for the enforcement of the source functions in \eqref{eq:disc_subpb_before} and \eqref{eq:disc_subpb_after}, nor for the computation of $\flux(\underline{\psi}_h^0)$ (via formula~\eqref{eq:flux_computation}) in the right-hand side of \eqref{eq:discrete_pb}.

\section{Preconditioner} \label{sec:precond}
In this section, we introduce the preconditioned iterative strategy for the solution of the algebraic realization of \eqref{eq:discrete_pb}.
\subsection{Algebraic setting}

For a fixed basis for the discrete HHO spaces, let $N_{\mathrm{B}} \eqbydef \operatorname{dim}\left(\mathbb{P}^k(\bfaces)\right)$.
Introduce the algebraic counterpart $\mathbf{L}_h \colon \mathbb{R}^{N_{\mathrm{B}}} \to \mathbb{R}^{N_{\mathrm{B}}}$ of the linear operator $\mathcal{L}_h$ defined by \eqref{eq:linear_operator}.
We recall that $\mathbf{L}_h$ can be viewed as a matrix in $\mathbb{R}^{N_{\mathrm{B}} \times N_{\mathrm{B}}}$, whose coefficients are not explicitly known, but whose application to a vector can be computed.
The algebraic realization of problem \eqref{eq:discrete_pb} reads: find $\lambdabf\in\mathbb{R}^{N_{\mathrm{B}}}$ such that
\begin{equation} \label{eq:algebraic_pb}
    \mathbf{L}_h(\lambdabf) = \mathbf{b},
\end{equation}
where $\mathbf{b}\in\mathbb{R}^{N_{\mathrm{B}}}$ is an algebraic representation of $\flux(\underline{\psi}_h^0) - \mathsf{g}_{\mathrm{N}}$ in the chosen basis.

\subsection{An approximate, sparse matrix}

The implicit system \eqref{eq:algebraic_pb} is solved using a preconditioned, flexible conjugate gradient (PFCG) algorithm. This section describes the construction of the preconditioner.
The use of a flexible version of the Krylov method is justified by the preconditioner being non-symmetric \cite{blaheta_gpcg-generalized_2002,bouwmeester_nonsymmetric_2015}.
Refer to \Cref{sec:discuss_symmetry} for a discussion about the symmetric property of the preconditioner.

Introducing $\left(\mathbf{e}_j\right)_{j=1, \dots, N_{\mathrm{B}}}$ the canonical basis of $\mathbb{R}^{N_{\mathrm{B}}}$, 
the $j^{\text{th}}$ column of $\mathbf{L}_h$ is given by the evaluation of $\mathbf{L}_h(\mathbf{e}_j)$. Note that $\mathbf{L}_h(\mathbf{e}_j)$ is a dense vector. Consequently, the explicit computation of $\mathbf{L}_h$ would result in a dense matrix. 
The preconditioner proposed here consists in the explicit computation of a \emph{sparse} matrix $\widetilde{\mathbf{L}}_h$ approximating $\mathbf{L}_h$. 
In a nutshell, while the computation of each column $\mathbf{L}_h(\mathbf{e}_j)$ involves solving two Laplacian problems in the whole domain, our approximation $\widetilde{\mathbf{L}}_h(\mathbf{e}_j)$ consists in solving those problems in a restricted neighbourhood of the $j^{th}$ DoF.

Let $j \in \{1, \dots, N_{\mathrm{B}} \}$ a fixed unknown index. 
Let $F^j \in \bfaces$ the face supporting the DoF associated to~$j$, and $T^j \in \cells$ the unique element owning $F^j$. 
Define $\cells^j \subset \cells$ a neighbourhood of $T^j$: roughly, a set of elements around $T^j$, including $T^j$.
Let $\Omega^j$ be the open, connected set such that $\overline{\Omega^j} \eqbydef \cup_{T\in\cells^j} \overline{T}$.
\Cref{fig:precond_mesh_nbh} illustrates (in grey) such a neighbourhood, where
$F^j$ is represented by a thick line and $T^j$ by a darker triangle. 
Relatively to $\Omega^j$, one defines the sets of interior and boundary faces by $(\ifaces)^j \eqbydef \{ F\in\ifaces \mid F\subset \Omega^j\}$ and $(\bfaces)^j \eqbydef \{ F\in\faces \mid F\subset \partial\Omega^j\}$, respectively.

Consider the operator $\widetilde{\mathcal{L}}_h^j$, counterpart of $\mathcal{L}_h$ in $\Omega^j$, in which $(\cells, \ifaces, \bfaces)$ is replaced with $(\cells^j, (\ifaces)^j, (\bfaces)^j)$ for the solution of the Laplacian subproblems, and $\bfaces$ is replaced with $\bfaces \cap (\bfaces)^j$ for the computation of the normal derivative. 
This latter point indicates that the normal derivative is computed on $\partial\Omega^j \cap \partial\Omega$ only. 
For both subproblems, homogeneous Dirichlet conditions are enforced on the part of $\partial\Omega^j$ interior to $\Omega$. 
Define $N_{\mathrm{B}}^j \eqbydef \operatorname{dim}\mathbb{P}^k((\bfaces)^j)$, $\left(\widetilde{\mathbf{e}}_j\right)_{j=1, \dots, N_{\mathrm{B}}^j}$ the canonical basis of $\mathbb{R}^{N_{\mathrm{B}}^j}$, and $\widetilde{\mathbf{L}}_h^j \colon \mathbb{R}^{N_{\mathrm{B}}^j} \to \mathbb{R}^{N_{\mathrm{B}}^j}$ the algebraic representation of $\widetilde{\mathcal{L}}_h^j$. 

Supposing that the $j^{\text{th}}$ basis function of $\mathbb{P}^k(\bfaces)$ corresponds to the $m^{\text{th}}$ basis function of $\mathbb{P}^k((\bfaces)^j)$,
we compute $\widetilde{\mathbf{L}}_h^j(\widetilde{\mathbf{e}}_m)$.
As $\widetilde{\mathbf{L}}_h^j(\widetilde{\mathbf{e}}_m)$ yields the normal derivative only on $\partial\Omega^j \cap \partial\Omega$, we build the final, approximate column $\widetilde{\mathbf{L}}_h(\mathbf{e}_j)$ from $\widetilde{\mathbf{L}}_h^j(\widetilde{\mathbf{e}}_m)$ by setting zero coefficients where the domain boundary is not covered by the neighbourhood boundary. 

We claim that $\widetilde{\mathbf{L}}_h(\mathbf{e}_j)$ yields a sparse approximation of $\mathbf{L}_h(\mathbf{e}_j)$. 
To support our claim, in 
\Cref{fig:precond_lambda_nbh,fig:precond_u_nbh} we plot, in iso-value curves, the solutions to the first and second Laplacian problems computed in $\Omega^j$. 
Moreover, \Cref{fig:precond_lambda_dom,fig:precond_u_dom} plot the solutions of the Laplacian subproblems computed in the whole domain. 
We observe that, owing to the homogeneity of the first equation and the homogeneous Dirichlet condition everywhere except on one face, the solution culminates on that face and decreases towards zero as it goes further away. 
Roughly speaking, our strategy consists in approximating the solution locally in $\Omega^j$ while imposing it to be zero in the rest of the domain.
Then, this first truncated solution is used as the source function in the second problem, also solved in $\Omega^j$. 
Here, note that the quantity of interest is not the solution itself, but its normal derivative on the boundary. 
In particular, the neighbourhood of the considered face concentrates the most information, insofar as the solution flattens as it goes away from it. 
In that sense, 
$\widetilde{\mathbf{L}}_h(\mathbf{e}_j)$ yields a reasonable approximation of $\mathbf{L}_h(\mathbf{e}_j)$.

\begin{figure}
  \begin{center}
  	\begin{subfigure}[t]{0.32\textwidth}
		\centering
		\includegraphics[trim=240 95 240 95, clip, width=0.6\textwidth]{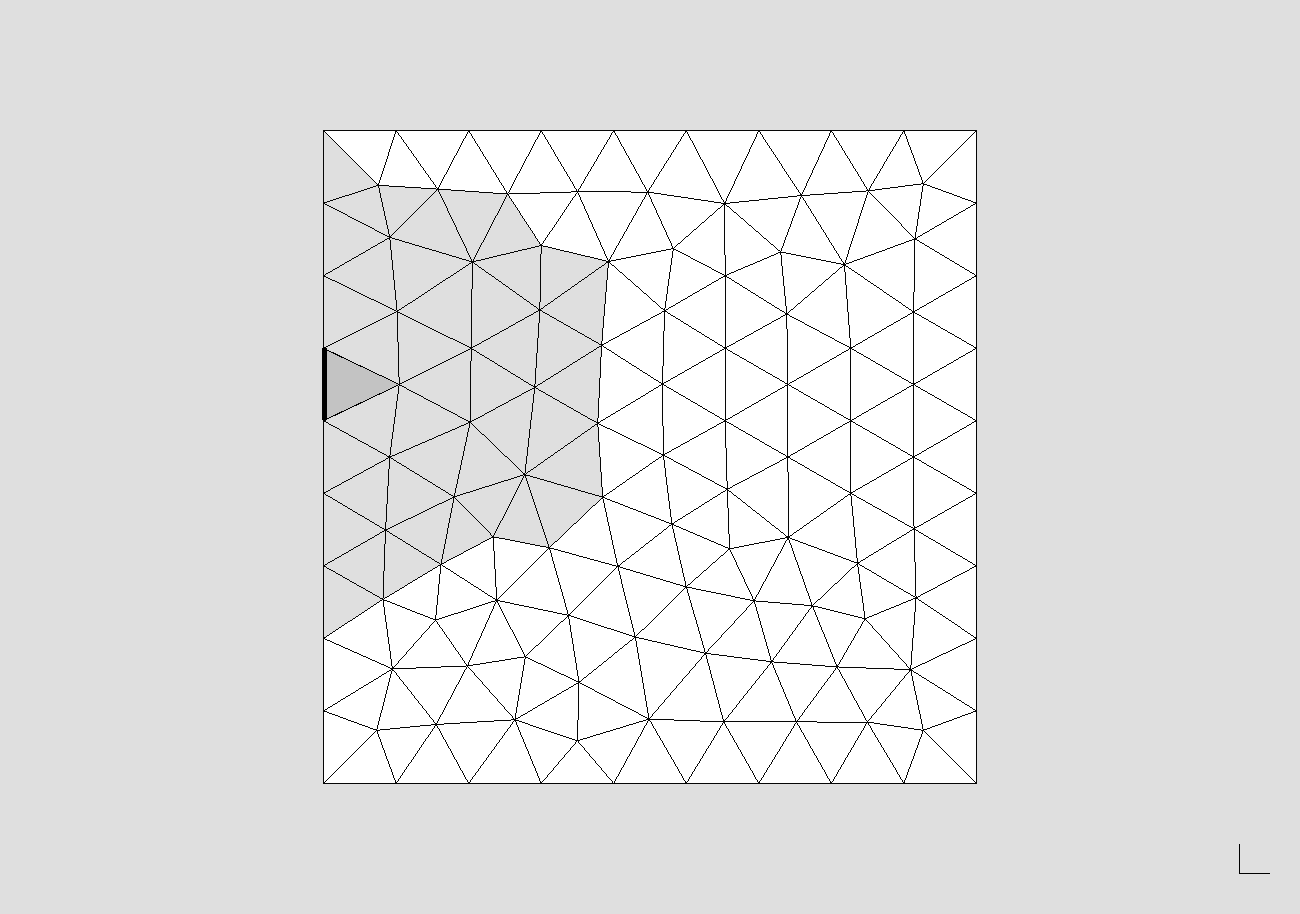}
		\caption{Edge $F^j$ (bold segment), element $T^j$ (dark grey) and neighbourhood $\Omega^{j}$ (light grey)}
	    \label{fig:precond_mesh_nbh}
    \end{subfigure}
    \hfill
    \begin{subfigure}[t]{0.32\textwidth} 
		\centering
		\includegraphics[trim=321 128 321 128, clip, width=0.6\textwidth]{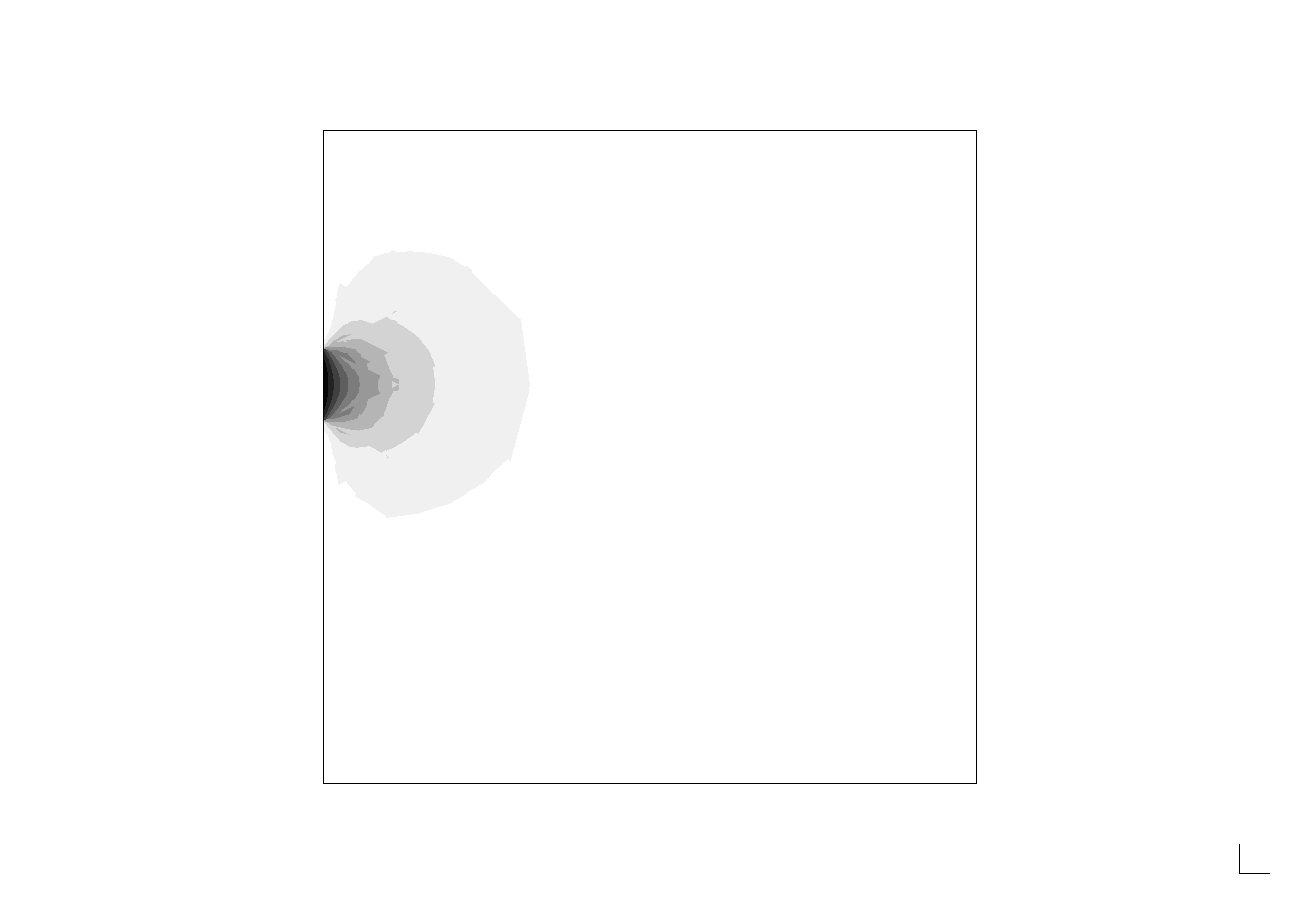}
		\caption{Solution of the first Laplacian problem computed in the neighbourhood}
	    \label{fig:precond_lambda_nbh}
    \end{subfigure}
    \hfill
    \begin{subfigure}[t]{0.32\textwidth} 
  		\centering
	    \includegraphics[trim=321 128 321 128, clip, width=0.6\textwidth]{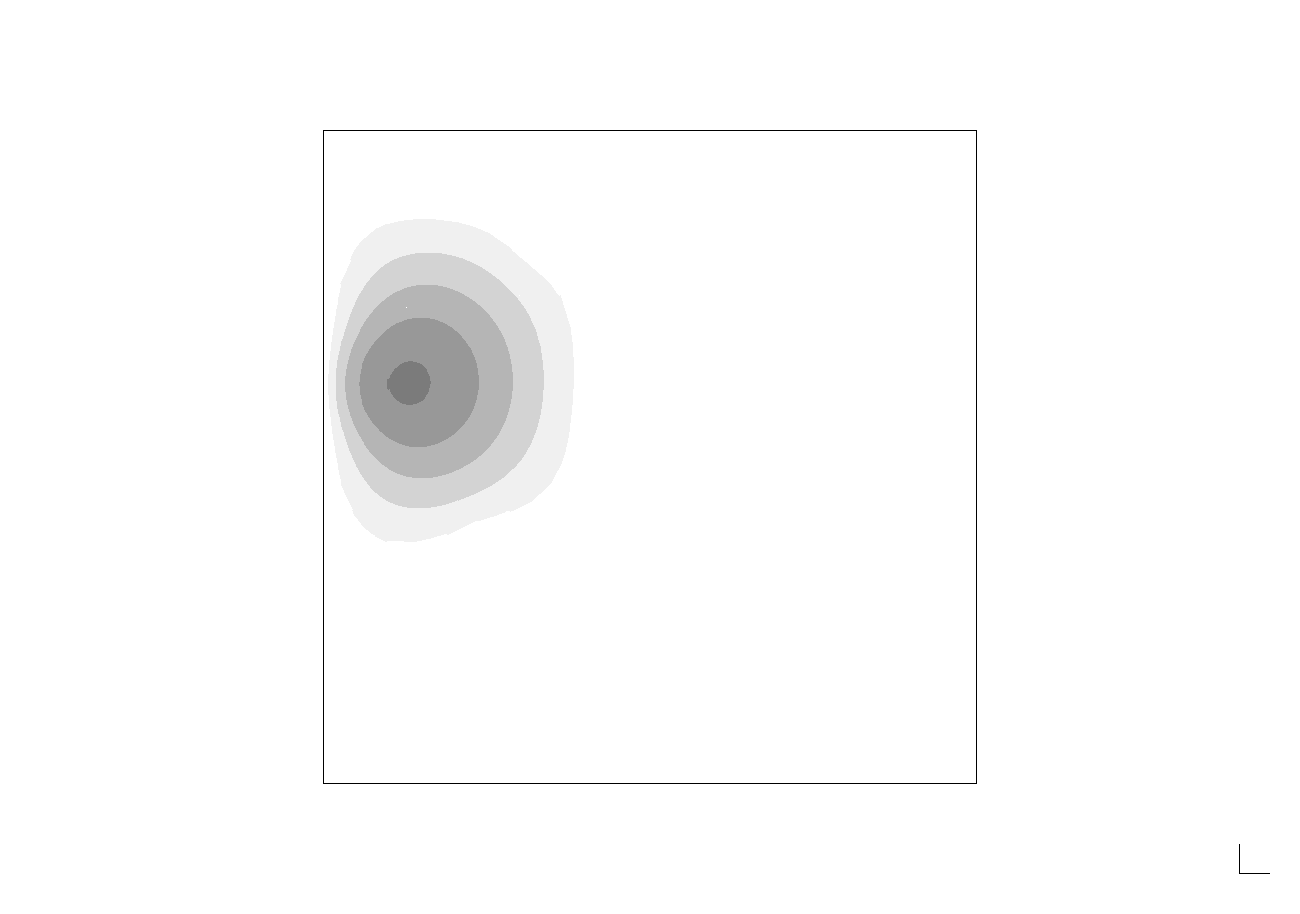}
	    \caption{Solution of the second Laplacian problem computed in the neighbourhood}
	    \label{fig:precond_u_nbh}
    \end{subfigure}
    \vskip\baselineskip
    \hspace{0.32\textwidth}
    \hfill
    \begin{subfigure}[t]{0.32\textwidth} 
		\centering
		\includegraphics[trim=321 128 321 128, clip, width=0.6\textwidth]{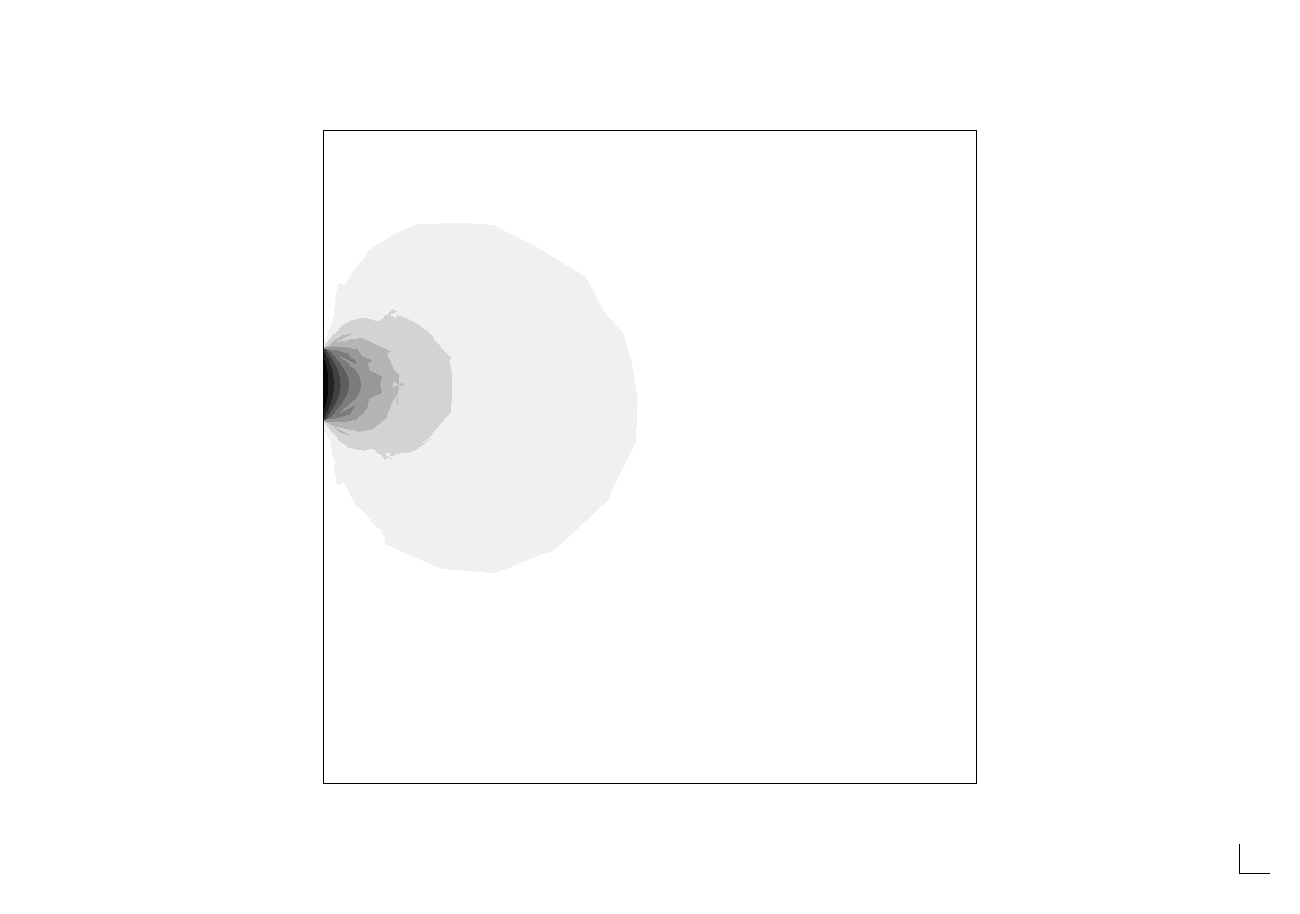}
		\caption{Solution of the first Laplacian problem computed in the whole domain}
	    \label{fig:precond_lambda_dom}
    \end{subfigure}
    \hfill
    \begin{subfigure}[t]{0.32\textwidth} 
  		\centering
	    \includegraphics[trim=321 128 321 128, clip, width=0.6\textwidth]{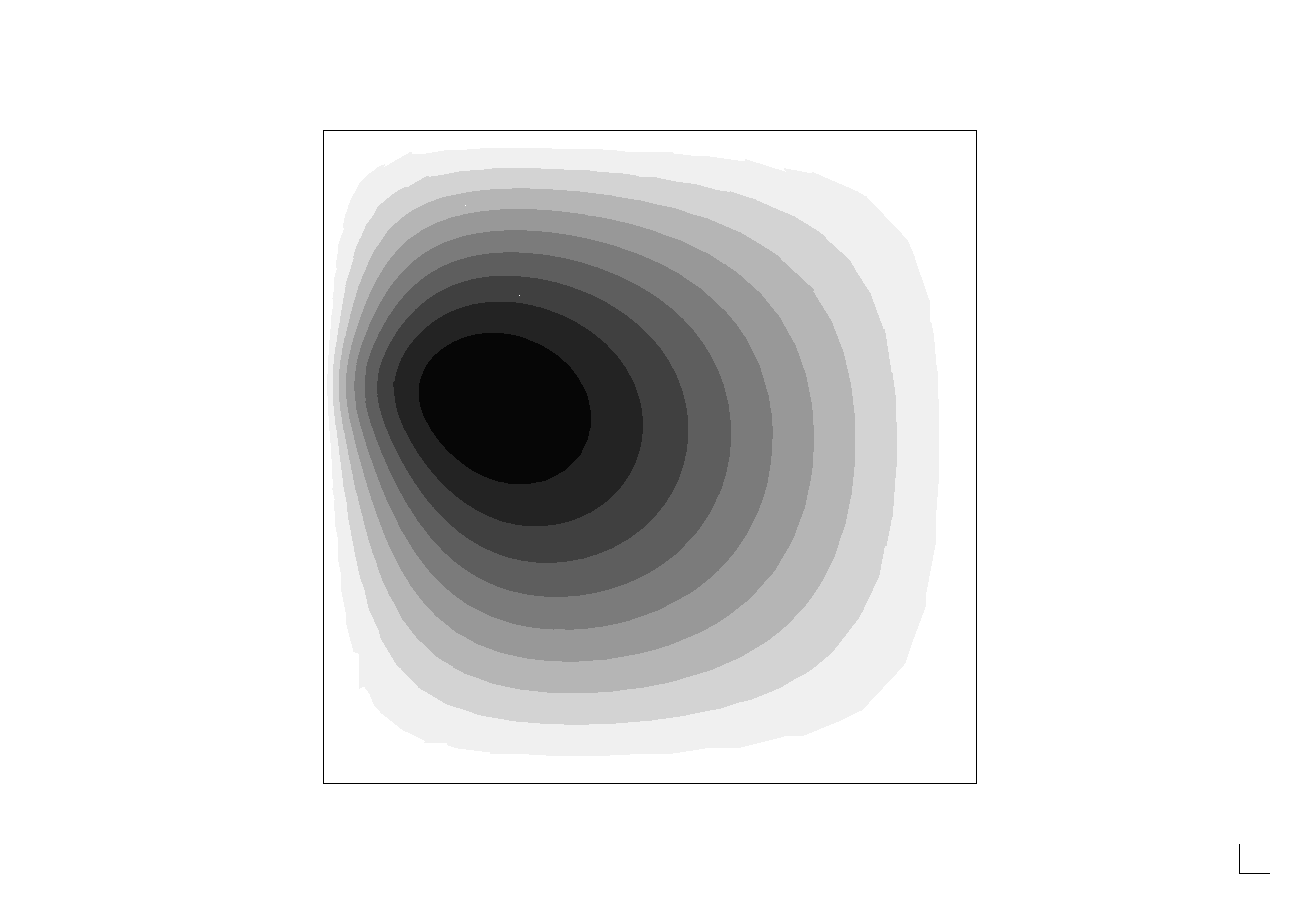}
	    \caption{Solution of the second Laplacian problem computed in the whole domain}
	    \label{fig:precond_u_dom}
    \end{subfigure}
    \caption{Illustration of the action of the preconditioner.}
    \label{fig:precond}
  \end{center}
\end{figure}

\subsection{Computation in practice} \label{sec:comput_in_practice}

Regarding the actual computation of $\widetilde{\mathbf{L}}_h$, one can note that the HHO matrix blocks used to solve the Laplacian problems in $\Omega^j$ and to compute the normal derivative can be extracted from the global ones. 
The sparsity of $\widetilde{\mathbf{L}}_h$ is controlled by the number of boundary faces in the chosen neighbourhood, while the accuracy of the approximation depends on the size of the neighbourhood. 
In particular, remark that choosing $\cells^j \eqbydef \cells$ for all~$j$ yields the actual matrix $\mathbf{L}_h$.
In practice, we construct $\cells^j$ by adding successive layers of neighbours around $T^j$. 
The number of layers is defined by the parameter $\alpha \in \mathbb{N}_0$. 
$\cells^j$ is defined as $\cells^j(\alpha)$, where $\cells^j(0) \eqbydef \{T^j\}$ and $\cells^j(\alpha) \eqbydef \cells^j(\alpha-1) \cup \{ T\in \cells \mid \exists T'\in \cells^j(\alpha-1) \text{ s.t. } \overline{T}\cap \overline{T'} \not= \emptyset \}$ for $\alpha \geq 1$. 
Note that this definition understands the neighbouring relationship as having at least one vertex in common. The neighbourhood represented in \Cref{fig:precond_mesh_nbh} corresponds to $\alpha=3$.

\subsection{Discussion on symmetry} \label{sec:discuss_symmetry}

Choosing a constant $\alpha$ for all $j$ allows to preserve a symmetric sparsity pattern.
Indeed, considering $(i, j) \in \{1, \dots, N_{\mathrm{B}} \}^2$, $(\widetilde{\mathbf{L}}_h)_{ij} \not= 0$ implies that $F^i\subset\partial\Omega^j$. Consequently, $(\widetilde{\mathbf{L}}_h)_{ji} \not= 0$ implies that, reciprocally, $F^j\subset\partial\Omega^i$.
However, symmetry itself is not preserved, inasmuch as $(\widetilde{\mathbf{L}}_h)_{ij}$ results from computations in the neighbourhood $\Omega^j$ whereas $(\widetilde{\mathbf{L}}_h)_{ji}$ results from computations in $\Omega^i \not= \Omega^j$.

To enforce the symmetry of the matrix $\mathbf{L}_h$, one must ensure that for all couple $(i, j) \in \{1, \dots, N_{\mathrm{B}} \}^2$, $\Omega^i = \Omega^j$. This requirement leads to a patch-based method: the set of boundary faces must be split into \emph{patches}, i.e.\ connected subsets of boundary faces. For each patch, one single neighbourhood must be defined to process all DoFs included in that patch. This yields a block-diagonal matrix $\widetilde{\mathbf{L}}_h$, one block corresponding to each patch. Algebraically, such a preconditioner is in fact an approximate block Jacobi preconditioner, insofar as each diagonal block approximates the corresponding diagonal block of the exact matrix $\mathbf{L}_h$.
This symmetric preconditioner was numerically tested. 
Even with the facts that a non-flexible Krylov method can be used and that the matrix $\widetilde{\mathbf{L}}_h$ is easier to factorize, the convergence of the method is significantly slower than with the unsymmetric preconditioner defined in \Cref{sec:comput_in_practice}.
Therefore, only the latter has been used in the numerical experiments of \Cref{sec:numerical_tests}.



\section{Numerical experiments} \label{sec:numerical_tests}

\subsection{Experimental setup}

The implicit system \eqref{eq:algebraic_pb} is solved iteratively by the PFCG method, using the preconditioning technique described in \Cref{sec:precond}.
Given the iterate $\widetilde{\lambdabf} \in \mathbb{R}^{N_{\mathrm{B}}}$, the corresponding residual vector is defined as $\mathbf{r} \eqbydef \mathbf{b} - \mathbf{L}_h(\widetilde{\lambdabf})$.
The PFCG algorithm stops when $||\mathbf{r}||_2/||\mathbf{b}||_2 < \varepsilon$, where $\varepsilon>0$ is a fixed tolerance and $||\cdot||_2$ denotes the Euclidean norm on $\mathbb{R}^{N_{\mathrm{B}}}$.
All experiments are conducted choosing $l = k$, but we stress that the same qualitative results are obtained with $l=k+1$.
In 2D, the linear systems corresponding to the Laplacian subproblems are solved by Cholesky factorization; in 3D, we use a $p$-multigrid algorithm on top of the algebraic multigrid method designed in \cite{di_pietro_algebraic_2021}.
In that latter case, the same tolerance $\varepsilon$ must be used to stop the Laplacian solvers and the global PFCG algorithm.
To apply the preconditioner, we solve $\widetilde{\mathbf{L}}_h$ using the BiCGSTAB algorithm with tolerance $\varepsilon$.
The computations are run on an 8-core processor (AMR M1 Pro) clocked at 3228 MHz.

\subsection{Preliminary: convergence of the discrete normal derivative} \label{sec:num_tests_normalder}

We first assert the validity of the discrete normal derivative proposed in \Cref{sec:normal_derivative} by evaluating its order of convergence.
Let $\Omega \eqbydef (0, 1)^2$ and $u\colon (x, y)\mapsto \sin(4\pi x)\sin(4\pi y)$ the manufactured solution of the boundary value problem~\eqref{eq:poisson_pb}, where $f$ and $g_\mathrm{D}$ are defined accordingly.
The problem is discretized on a sequence of successively refined meshes, and we consider $\flux(\underline{u}_h)$ computed by \eqref{eq:flux_computation_condensed}.
The exact normal derivative $\normalder{u}$ is known, and we assess the relative $L^2$-error $\|\normalder{u} - \flux(\underline{u}_h)\|_{L^2(\partial\Omega)} / \|\normalder{u}\|_{L^2(\partial\Omega)}$ for $k\in\{0, 1, 2, 3\}$.
\Cref{fig:normalder_convergence:cart,fig:normalder_convergence:delaunay} present the experimental results on uniform Cartesian meshes and on unstructured, triangular Delaunay meshes, respectively. 
These experiments assess a convergence in $\mathcal{O}(h^{k+1})$ (also observed for $l=k+1$), which is in line with the literature on finite elements \cite[Cor.\ 6.1]{melenk_quasi_2012}. Indeed, the convergence order of the normal derivative relies on that of the Laplacian scheme: if the discrete solution of the Laplacian scheme converges in $\mathcal{O}(h^p)$, with $p>1$, then the discrete normal derivative on the boundary converges in $\mathcal{O}(h^{p-1})$.
Transposed to HHO, since $u_h \eqbydef p_h^{k+1}\underline{u}_h$ converges in $\mathcal{O}(h^{k+2})$, it is then expected that $\flux(\underline{u}_h)$ converge in $\mathcal{O}(h^{k+1})$.
Note that in the case of a non-convex domain with re-entrant corners, this convergence order is expected to be reduced due to the corner singularities \cite{pfefferer_normalder_2019}. One way to mitigate this issue is the use of graded meshes (cf. \cite{pfefferer_normalder_2019}).

\begin{figure}
  \begin{center}
    \begin{subfigure}[t]{0.48\textwidth}
     \centering
     \begin{tikzpicture}[scale=1]
      \begin{loglogaxis}[
          width=\convergencePlotWidth,
          height=\convergencePlotHeight,
          ymajorgrids,
          xlabel=$h$,
          ylabel={$L^2$-error},
          ymin=1e-11, ymax=1e-0,
          cycle list name=convergence_k,
          legend pos=outer north east,
          legend cell align={left},
        ]
        \logLogSlopeTriangle{0.25}{0.16}{0.76}{1}{black};
        \addplot coordinates {
            (8.838835e-02, 2.24e-01) 
            (4.419417e-02, 1.13e-01) 
            (2.209709e-02, 5.66e-02) 
            (1.104854e-02, 2.83e-02) 
            (5.524272e-03, 1.42e-02) 
            (2.762136e-03, 7.09e-03) 
        };
        \logLogSlopeTriangle{0.25}{0.16}{0.54}{2}{blue};
        \addplot coordinates {
            (8.838835e-02, 2.96e-02) 
            (4.419417e-02, 6.44e-03) 
            (2.209709e-02, 1.49e-03) 
            (1.104854e-02, 3.63e-04) 
            (5.524272e-03, 9.00e-05) 
            (2.762136e-03, 2.25e-05) 
        };
        \logLogSlopeTriangle{0.25}{0.16}{0.30}{3}{red};
        \addplot coordinates {
            (8.838835e-02, 1.85e-03) 
            (4.419417e-02, 2.03e-04) 
            (2.209709e-02, 2.42e-05) 
            (1.104854e-02, 2.99e-06) 
            (5.524272e-03, 3.73e-07) 
            (2.762136e-03, 4.66e-08) 
        };
        \logLogSlopeTriangle{0.25}{0.16}{0.10}{4}{brown!60!black};
        \addplot coordinates {
            (8.838835e-02, 4.36e-04) 
            (4.419417e-02, 2.63e-05) 
            (2.209709e-02, 1.63e-06) 
            (1.104854e-02, 1.01e-07) 
            (5.524272e-03, 6.34e-09) 
            (2.762136e-03, 3.96e-10) 
        };
        \legend{$k=0$, $k=1$, $k=2$, $k=3$};
      \end{loglogaxis}
      \end{tikzpicture}
        \caption{Uniform, Cartesian meshes.}
        \label{fig:normalder_convergence:cart}
    \end{subfigure}
    \begin{subfigure}[t]{0.48\textwidth}
     \centering
     \begin{tikzpicture}[scale=1]
      \begin{loglogaxis}[
          width=\convergencePlotWidth,
          height=\convergencePlotHeight,
          ymajorgrids,
          xlabel=$h$,
          ymin=1e-11, ymax=1e-0,
          cycle list name=convergence_k,
        ]
        \logLogSlopeTriangle{0.25}{0.16}{0.76}{1}{black};
        \addplot coordinates {
            (7.513969e-02, 2.12e-01) 
            (3.661076e-02, 1.10e-01) 
            (2.084383e-02, 5.58e-02) 
            (1.053829e-02, 2.81e-02) 
            (5.223640e-03, 1.41e-02) 
            (2.629384e-03, 7.08e-03) 
        };
        \logLogSlopeTriangle{0.25}{0.16}{0.54}{2}{blue};
        \addplot coordinates {
            (7.513969e-02, 2.16e-02) 
            (3.661076e-02, 5.50e-03) 
            (2.084383e-02, 1.46e-03) 
            (1.053829e-02, 3.67e-04) 
            (5.223640e-03, 9.25e-05) 
            (2.629384e-03, 2.29e-05) 
        };
        \logLogSlopeTriangle{0.25}{0.16}{0.30}{3}{red};
        \addplot coordinates {
            (7.513969e-02, 1.59e-03) 
            (3.661076e-02, 2.11e-04) 
            (2.084383e-02, 2.72e-05) 
            (1.053829e-02, 3.55e-06) 
            (5.223640e-03, 4.54e-07) 
            (2.629384e-03, 6.49e-08) 
        };
        \logLogSlopeTriangle{0.25}{0.16}{0.10}{4}{brown!60!black};
        \addplot coordinates {
            (7.513969e-02, 2.14e-04) 
            (3.661076e-02, 1.69e-05) 
            (2.084383e-02, 1.02e-06) 
            (1.053829e-02, 6.54e-08) 
            (5.223640e-03, 4.15e-09) 
            (2.629384e-03, 3.07e-10) 
        };
      \end{loglogaxis}
      \end{tikzpicture}
        \caption{Delaunay meshes.}
        \label{fig:normalder_convergence:delaunay}
    \end{subfigure}
    \caption{Convergence of the discrete normal derivative $\flux(\underline{u}_h)$ for a smooth solution.}
    \label{fig:normalder_convergence}
  \end{center}
\end{figure}
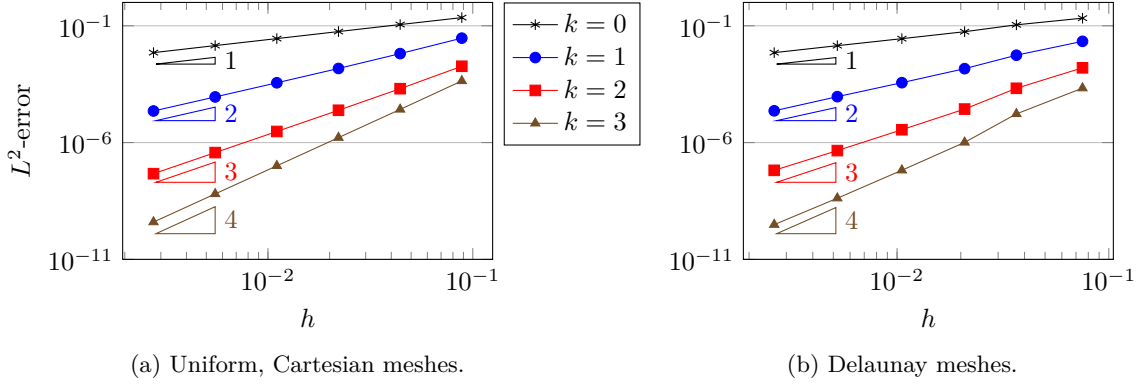

\subsection{Convergence rate of the scheme}

On the unit square, we consider $\psi(x, y) = x\sin(\pi y)e^{-xy}$ and its corresponding $\omega=-\Delta \psi$, the manufactured solution of \eqref{eq:mixed}, with $\mathsf{f}, \mathsf{g}_\mathrm{D}$ and $\mathsf{g}_\mathrm{N}$ determined accordingly.
The square is decomposed into a sequence of successively finer meshes, and the numerical scheme is applied for $k\in\{0, 1, 2, 3\}$. 
\Cref{fig:scheme_convergence} presents the evolution of the relative $L^2$-error achieved by the approximation $(\omega_h, \psi_h)$ with respect to $(\omega, \psi)$, on a sequence of Cartesian meshes.
Similarly, \Cref{fig:scheme_convergence:polymesh} shows the analogous results obtained with a sequence of polygonal meshes, each constructed from an initial Cartesian mesh by agglomeration and face collapsing.
One such mesh is represented in \Cref{fig:poly_mesh0}.
We empirically interpret these results in light of Glowinski \emph{et al.}'s theoretical findings in \cite{glowinski_numerical_1979} for the continuous high-order FEM.
According to \cite[Section~3.3]{glowinski_numerical_1979}, since $\psi$ is regular enough, the FEM solution $(\omega_h^\textrm{FEM}, \psi_h^\textrm{FEM})$ of order $p \geq 3$ verifies
\begin{equation} \label{eq:fem_error_estimate}
    \| \psi_h^\textrm{FEM} - \psi \|_{L^2(\Omega)} + h^2\| \omega_h^\textrm{FEM} - \omega \|_{L^2(\Omega)} \leq C h^{p+1} \|\psi\|_{H^{p+1}(\Omega)},
\end{equation}
where $C$ is a constant independent of $h$ and $\psi$.
In our HHO setting, the above error estimate would translate in the following one: if $k \geq 2$, then
\begin{equation} \label{eq:fem2hho_error_estimate}
    \| \psi_h - \psi \|_{L^2(\Omega)} + h^2\| \omega_h - \omega \|_{L^2(\Omega)} \leq C h^{k+2} \|\psi\|_{H^{k+2}(\Omega)}.
\end{equation}
The experimental results of \Cref{fig:scheme_convergence} validate the estimate \eqref{eq:fem2hho_error_estimate}, not only for $k \geq 2$, but also for $k\in \{0, 1\}$.
More precisely, while \Cref{fig:scheme_convergence:psi} shows that the estimate is sharp for $\psi$ (i.e.\ $\mathcal{O}(h^{k+2})$), \Cref{fig:scheme_convergence:omega} shows, for $\omega$, a faster convergence than indicated by \eqref{eq:fem2hho_error_estimate}.
Namely, the convergence seems to be $\mathcal{O}(h^{k+1})$ for $k=0$ and $\mathcal{O}(h^{k+\nicefrac12})$ for $k\geq 1$.
The experiments of \Cref{fig:scheme_convergence:polysol} on polygonal meshes exhibit the same convergence orders, except for $\omega_h$ with $k=0$, where the error estimate reduces to $\mathcal{O}(h^{k+\nicefrac12})$, in accordance with the other values of $k$.
Based on these experiments, we conjecture that the following error estimate is sharp for all $k\geq 0$:
\begin{equation} \label{eq:hho_error_estimate}
    \| \psi_h - \psi \|_{L^2(\Omega)} + h^{\nicefrac32}\| \omega_h - \omega \|_{L^2(\Omega)} \leq C h^{k+2} \|\psi\|_{H^{k+2}(\Omega)}.
\end{equation}
Nonetheless, we keep in mind that $k=0$ seems to be a special case for $\omega_h$, where superconvergence may be observed: besides the order 1 obtained on Cartesian meshes (\Cref{fig:scheme_convergence:omega}), we observe the order 2 if we choose as an exact solution the polynomial function $\psi(x, y) = x^4(x - 1)^2 y^4(y - 1)^2$, as illustrated in \Cref{fig:scheme_convergence:polysol}.

\begin{remark}
    (Comparison with \cite{dong_hybrid_2022})
    The HHO discretizations of \cite{dong_hybrid_2022} rely on the primal formulation \eqref{eq:biharmonic} of the equation.
    They hinge on the local approximation space $\mathbb{P}^{k+2}(T)\times\brokenPolyF{T}{k+1}\times\brokenPolyF{T}{k}$ in 2D (resp.\ $\mathbb{P}^{k+2}(T)\times\brokenPolyF{T}{k+2}\times\brokenPolyF{T}{k}$ in 3D), where the associated set of DoFs aims at approximating $(\psi_{|T}, \psi_{|\partial T}, \nabla\psi\cdot \mathbf{n}_{\partial T})$. 
    With this setting, the methods achieve a convergence order of $k+3$ in $L^2$-norm.
    In comparison, the present method requires the space $\mathbb{P}^{k+1}(T)\times\brokenPolyF{T}{k+1}$ to obtain the same convergence order, i.e.\ one less polynomial degree for the cell unknowns and no DoF dedicated to the approximation of $\nabla\psi\cdot \mathbf{n}_{\partial T}$.
    However, if our method is structurally lighter in DoFs than \cite{dong_hybrid_2022}, one cannot conclude regarding the computational cost and overall efficiency. 
    Attempting a fair comparison necessitates to account for the fact that, in spite of the lower number of unknowns and the existence of fast solvers for second-order elliptic systems, two such systems have to be solved at each iteration of the iterative process. In comparison, the monolithic system arising from \cite{dong_hybrid_2022}, although larger and exhibiting a more challenging structure, has to be solved only once. If an iterative solver is used, notwithstanding the preconditioner, no inner system has to be solved at each iteration, which suggests a lower cost per iteration than the present method.
    However, for large-scale problems, scalability becomes essential: on top of the cost per iteration, convergence speed and massive parallelization become crucial criteria. In this context, the efficiency of a method strongly depends on the availability of an efficient preconditioner.
\end{remark}

\begin{figure}
  \begin{center}
    \begin{subfigure}[t]{0.48\textwidth}
  	 \centering
     \begin{tikzpicture}[scale=1]
      \begin{loglogaxis}[
          width=\convergencePlotWidth,
          height=\convergencePlotHeight,
          ymajorgrids,
          xlabel=$h$,
          ylabel={$L^2$-error},
          ytick={1e-0, 1e-3, 1e-6, 1e-9},
          cycle list name=convergence_k,
          legend pos=outer north east,
          legend cell align={left},
        ]
        \logLogSlopeTriangle{0.28}{0.18}{0.74}{1.0}{black};
        \addplot coordinates {
            (8.838835e-02, 6.38e-02) 
            (4.419417e-02, 3.00e-02) 
            (2.209709e-02, 1.45e-02) 
            (1.104854e-02, 7.15e-03) 
            (5.524272e-03, 3.55e-03) 
        };
        \logLogSlopeTriangle{0.28}{0.18}{0.49}{1.5}{blue};
        \addplot coordinates {
            (8.838835e-02, 1.37e-03) 
            (4.419417e-02, 3.95e-04) 
            (2.209709e-02, 1.24e-04) 
            (1.104854e-02, 4.10e-05) 
            (5.524272e-03, 1.40e-05) 
        };
        \logLogSlopeTriangle{0.28}{0.18}{0.25}{2.5}{red};
        \addplot coordinates {
            (8.838835e-02, 9.79e-05) 
            (4.419417e-02, 1.48e-05) 
            (2.209709e-02, 2.38e-06) 
            (1.104854e-02, 3.98e-07) 
            (5.524272e-03, 6.83e-08) 
        };
        \logLogSlopeTriangle{0.28}{0.18}{0.05}{3.5}{brown!60!black};
        \addplot coordinates {
            (8.838835e-02, 4.72e-06) 
            (4.419417e-02, 5.29e-07) 
            (2.209709e-02, 6.16e-08) 
            (1.104854e-02, 7.39e-09) 
            (5.524272e-03, 9.03e-10) 
        };
        \legend{$k=0$, $k=1$, $k=2$, $k=3$};
      \end{loglogaxis}
      \end{tikzpicture}
        \caption{$\omega$}
      \label{fig:scheme_convergence:omega}
    \end{subfigure}
    \hfill
  \begin{subfigure}[t]{0.48\textwidth}
   \centering
     \begin{tikzpicture}[scale=1]
      \begin{loglogaxis}[
          width=\convergencePlotWidth,
          height=\convergencePlotHeight,
          ymajorgrids,
          xlabel=$h$,
          ytick={1e-0, 1e-3, 1e-6, 1e-9, 1e-12},
          cycle list name=convergence_k,
          legend pos=outer north east,
          legend cell align={left},
        ]
        \logLogSlopeTriangle{0.28}{0.18}{0.68}{2}{black};
        \addplot coordinates {
            (8.838835e-02, 1.10e-02) 
            (4.419417e-02, 2.72e-03) 
            (2.209709e-02, 6.78e-04) 
            (1.104854e-02, 1.69e-04) 
            (5.524272e-03, 4.23e-05) 
        };
        \logLogSlopeTriangle{0.28}{0.18}{0.38}{3}{blue};
        \addplot coordinates {
            (8.838835e-02, 6.25e-05) 
            (4.419417e-02, 6.23e-06) 
            (2.209709e-02, 6.85e-07) 
            (1.104854e-02, 8.10e-08) 
            (5.524272e-03, 9.88e-09) 
        };
        \logLogSlopeTriangle{0.28}{0.18}{0.15}{4}{red};
        \addplot coordinates {
            (8.838835e-02, 8.74e-07) 
            (4.419417e-02, 5.16e-08) 
            (2.209709e-02, 3.17e-09) 
            (1.104854e-02, 1.97e-10) 
            (5.524272e-03, 1.23e-11) 
        };
        \logLogSlopeTriangle{0.50}{0.18}{0.07}{5}{brown!60!black};
        \addplot coordinates {
            (8.838835e-02, 3.25e-08) 
            (4.419417e-02, 1.04e-09) 
            (2.209709e-02, 3.27e-11) 
            (1.104854e-02, 1.12e-12) 
            (5.524272e-03, 1.08e-12) 
        };
      \end{loglogaxis}
      \end{tikzpicture}
      \caption{$\psi$}
      \label{fig:scheme_convergence:psi}
    \end{subfigure}
    \caption{Convergence in $L^2$-norm of the discrete solution $(\omega_h, \psi_h)$ with respect to the exact solution $\psi(x, y) = x\sin(\pi y)e^{-xy}$. Square domain discretized by Cartesian meshes.}
    \label{fig:scheme_convergence}
  \end{center}
\end{figure}
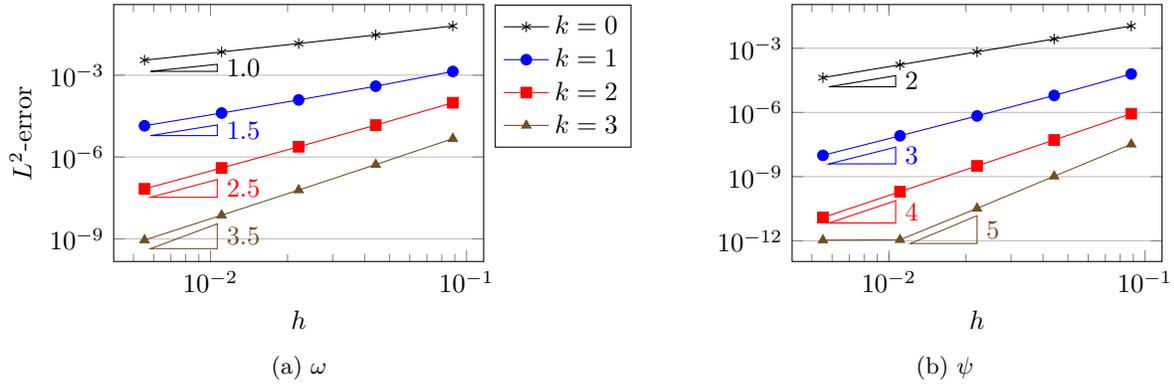

\begin{figure}
  \begin{center}
    \begin{subfigure}[t]{0.48\textwidth}
  	 \centering
     \begin{tikzpicture}[scale=1]
      \begin{loglogaxis}[
          width=\convergencePlotWidth,
          height=\convergencePlotHeight,
          ymajorgrids,
          xlabel=$h$,
          ylabel={$L^2$-error},
          ytick={1e-0, 1e-3, 1e-6, 1e-9},
          cycle list name=convergence_k,
          legend pos=outer north east,
          legend cell align={left},
        ]
        \logLogSlopeTriangle{0.28}{0.18}{0.80}{0.5}{black};
        \addplot coordinates {
            (1.33e-01    , 4.85e-02) 
            (6.63e-02    , 2.92e-02) 
            (3.31e-02    , 2.11e-02) 
            (1.66e-02    , 1.35e-02) 
            (8.29e-03    , 9.17e-03) 
        };
        \logLogSlopeTriangle{0.30}{0.18}{0.53}{1.5}{blue};
        \addplot coordinates {
            (1.33e-01    , 2.28e-03) 
            (6.63e-02    , 8.73e-04) 
            (3.31e-02    , 3.11e-04) 
            (1.75e-02    , 1.16e-04) 
            (8.73e-03    , 3.75e-05) 
        };
        \logLogSlopeTriangle{0.30}{0.18}{0.25}{2.5}{red};
        \addplot coordinates {
            (1.33e-01    , 1.45e-04) 
            (6.63e-02    , 2.10e-05) 
            (3.31e-02    , 3.75e-06) 
            (1.75e-02    , 6.79e-07) 
            (8.29e-03    , 1.16e-07) 
        };
        \logLogSlopeTriangle{0.28}{0.18}{0.05}{3.5}{brown!60!black};
        \addplot coordinates {
            (1.33e-01    , 7.64e-06) 
            (6.63e-02    , 1.68e-06) 
            (3.31e-02    , 2.19e-07) 
            (1.66e-02    , 2.57e-08) 
            (8.29e-03    , 2.84e-09) 
        };
        \legend{$k=0$, $k=1$, $k=2$, $k=3$};
      \end{loglogaxis}
      \end{tikzpicture}
        \caption{$\omega$}
      \label{fig:scheme_convergence:polymesh:omega}
    \end{subfigure}
    \hfill
  \begin{subfigure}[t]{0.48\textwidth}
   \centering
     \begin{tikzpicture}[scale=1]
      \begin{loglogaxis}[
          width=\convergencePlotWidth,
          height=\convergencePlotHeight,
          ymajorgrids,
          xlabel=$h$,
          ytick={1e-0, 1e-3, 1e-6, 1e-9, 1e-12},
          cycle list name=convergence_k,
          legend pos=outer north east,
          legend cell align={left},
        ]
        \logLogSlopeTriangle{0.29}{0.18}{0.68}{2}{black};
        \addplot coordinates {
            (1.33e-01    , 1.01e-02) 
            (6.63e-02    , 2.43e-03) 
            (3.31e-02    , 6.72e-04) 
            (1.66e-02    , 1.85e-04) 
            (8.29e-03    , 4.41e-05) 
        };
        \logLogSlopeTriangle{0.30}{0.18}{0.40}{3}{blue};
        \addplot coordinates {
            (1.33e-01    , 1.12e-04) 
            (6.63e-02    , 1.34e-05) 
            (3.31e-02    , 1.69e-06) 
            (1.75e-02    , 2.09e-07) 
            (8.73e-03    , 2.61e-08) 
        };
        \logLogSlopeTriangle{0.29}{0.18}{0.17}{4}{red};
        \addplot coordinates {
            (1.33e-01    , 3.60e-06) 
            (6.63e-02    , 1.88e-07) 
            (3.31e-02    , 1.17e-08) 
            (1.75e-02    , 7.45e-10) 
            (8.29e-03    , 4.63e-11) 
        };
        \logLogSlopeTriangle{0.50}{0.18}{0.06}{5}{brown!60!black};
        \addplot coordinates {
            (1.33e-01    , 6.72e-08) 
            (6.63e-02    , 2.33e-09) 
            (3.31e-02    , 7.55e-11) 
            (1.66e-02    , 2.42e-12) 
            (8.29e-03    , 2.42e-12) 
        };
      \end{loglogaxis}
      \end{tikzpicture}
      \caption{$\psi$}
      \label{fig:scheme_convergence:polymesh:psi}
    \end{subfigure}
    \caption{Convergence in $L^2$-norm of the discrete solution $(\omega_h, \psi_h)$ with respect to the exact solution $\psi(x, y) = x\sin(\pi y)e^{-xy}$. Square domain discretized by polygonal meshes.}
    \label{fig:scheme_convergence:polymesh}
  \end{center}
\end{figure}
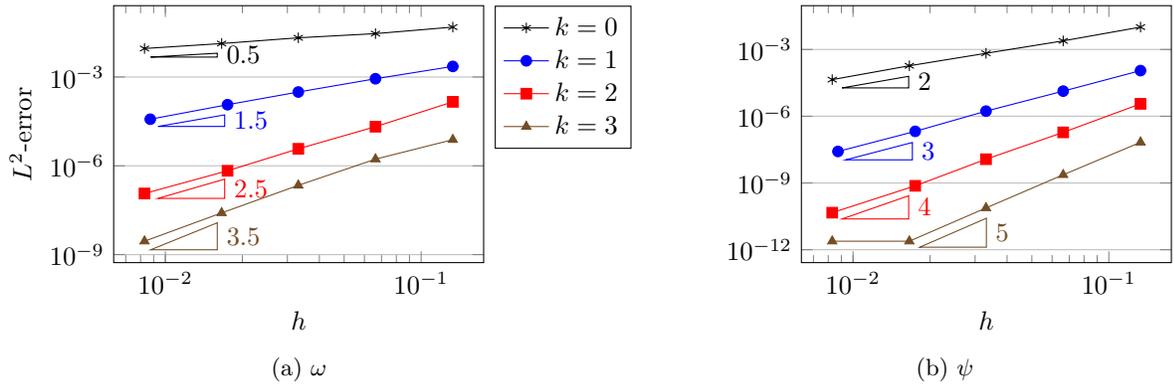

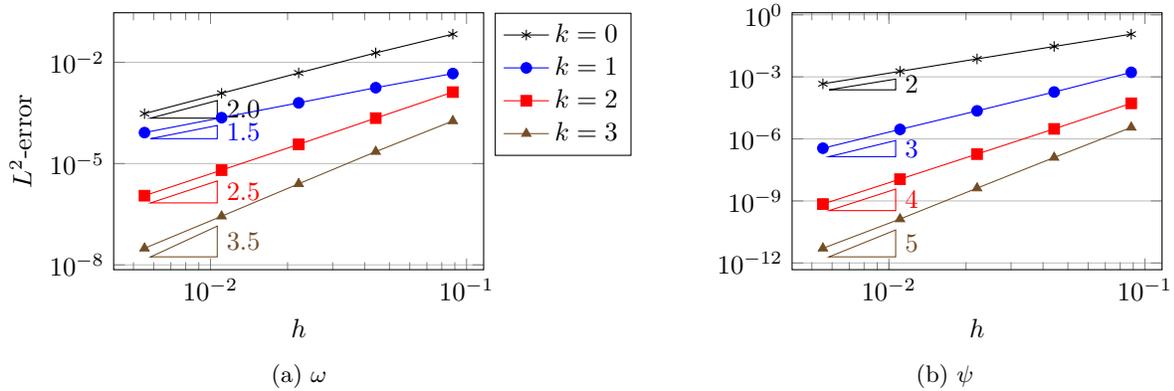
\begin{figure}
  \begin{center}
    \begin{subfigure}[t]{0.48\textwidth}
  	 \centering
     \begin{tikzpicture}[scale=1]
      \begin{loglogaxis}[
          width=\convergencePlotWidth,
          height=\convergencePlotHeight,
          ymajorgrids,
          xlabel=$h$,
          ylabel={$L^2$-error},
          ytick={1e-0, 1e-2, 1e-5, 1e-8},
          cycle list name=convergence_k,
          legend pos=outer north east,
          legend cell align={left},
        ]
        \logLogSlopeTriangle{0.28}{0.18}{0.59}{2.0}{black};
        \addplot coordinates {
            (8.838835e-02, 6.81e-02) 
            (4.419417e-02, 1.90e-02) 
            (2.209709e-02, 4.84e-03) 
            (1.104854e-02, 1.21e-03) 
            (5.524272e-03, 3.02e-04) 
        };
        \logLogSlopeTriangle{0.28}{0.18}{0.51}{1.5}{blue};
        \addplot coordinates {
            (8.838835e-02, 4.64e-03) 
            (4.419417e-02, 1.77e-03) 
            (2.209709e-02, 6.30e-04) 
            (1.104854e-02, 2.29e-04) 
            (5.524272e-03, 8.28e-05) 
        };
        \logLogSlopeTriangle{0.28}{0.18}{0.26}{2.5}{red};
        \addplot coordinates {
            (8.838835e-02, 1.31e-03) 
            (4.419417e-02, 2.22e-04) 
            (2.209709e-02, 3.71e-05) 
            (1.104854e-02, 6.40e-06) 
            (5.524272e-03, 1.12e-06) 
        };
        \logLogSlopeTriangle{0.28}{0.18}{0.05}{3.5}{brown!60!black};
        \addplot coordinates {
            (8.838835e-02, 1.83e-04) 
            (4.419417e-02, 2.26e-05) 
            (2.209709e-02, 2.52e-06) 
            (1.104854e-02, 2.76e-07) 
            (5.524272e-03, 3.10e-08) 
        };
        \legend{$k=0$, $k=1$, $k=2$, $k=3$};
      \end{loglogaxis}
      \end{tikzpicture}
        \caption{$\omega$}
      \label{fig:scheme_convergence:polysol:omega}
    \end{subfigure}
    \hfill
  \begin{subfigure}[t]{0.48\textwidth}
   \centering
     \begin{tikzpicture}[scale=1]
      \begin{loglogaxis}[
          width=\convergencePlotWidth,
          height=\convergencePlotHeight,
          ymajorgrids,
          xlabel=$h$,
          ytick={1e-0, 1e-3, 1e-6, 1e-9, 1e-12},
          cycle list name=convergence_k,
          legend pos=outer north east,
          legend cell align={left},
        ]
        \logLogSlopeTriangle{0.28}{0.18}{0.7}{2}{black};
        \addplot coordinates {
            (8.838835e-02, 1.16e-01) 
            (4.419417e-02, 2.95e-02) 
            (2.209709e-02, 7.40e-03) 
            (1.104854e-02, 1.85e-03) 
            (5.524272e-03, 4.63e-04) 
        };
        \logLogSlopeTriangle{0.28}{0.18}{0.44}{3}{blue};
        \addplot coordinates {
            (8.838835e-02, 1.65e-03) 
            (4.419417e-02, 1.85e-04) 
            (2.209709e-02, 2.27e-05) 
            (1.104854e-02, 2.88e-06) 
            (5.524272e-03, 3.54e-07) 
        };
        \logLogSlopeTriangle{0.28}{0.18}{0.23}{4}{red};
        \addplot coordinates {
            (8.838835e-02, 5.25e-05) 
            (4.419417e-02, 3.05e-06) 
            (2.209709e-02, 1.84e-07) 
            (1.104854e-02, 1.13e-08) 
            (5.524272e-03, 7.05e-10) 
        };
        \logLogSlopeTriangle{0.28}{0.18}{0.05}{5}{brown!60!black};
        \addplot coordinates {
            (8.838835e-02, 3.60e-06) 
            (4.419417e-02, 1.26e-07) 
            (2.209709e-02, 4.15e-09) 
            (1.104854e-02, 1.33e-10) 
            (5.524272e-03, 5.13e-12) 
        };
      \end{loglogaxis}
      \end{tikzpicture}
      \caption{$\psi$}
      \label{fig:scheme_convergence:polysol:psi}
    \end{subfigure}
    \caption{Convergence in $L^2$-norm of the discrete solution $(\omega_h, \psi_h)$ with respect to the exact solution $\psi(x, y) = x^4(x - 1)^2 y^4(y - 1)^2$. Square domain discretized by Cartesian meshes.}
    \label{fig:scheme_convergence:polysol}
  \end{center}
\end{figure}

\begin{figure}
    \centering
    \includegraphics[scale=0.25, trim = 40mm 90mm 40mm 85mm, clip]{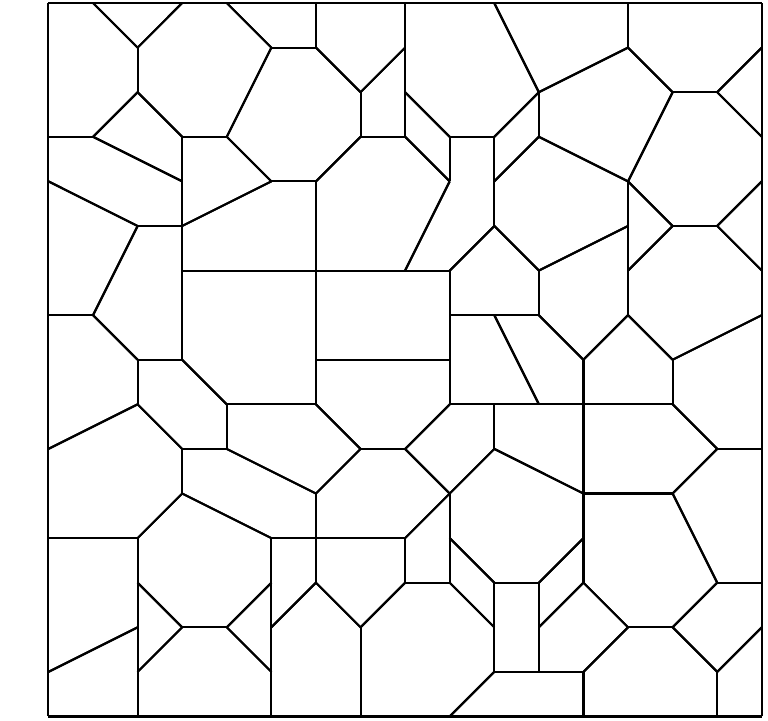}
    \caption{Example of a polygonal mesh of the unit square $\Omega = (0, 1)^2$.}
    \label{fig:poly_mesh0}
\end{figure}

\subsection{Preconditioner} \label{sec:num_exp_precond}

We evaluate in this section the efficiency of the preconditioner designed in \Cref{sec:precond}. We recall that the parameter $\alpha \in \mathbb{N}_0$ is an indicator of the size of the neighbourhoods used for the approximation of the matrix. 

\subsubsection{Convergence vs.\ cost: the parameter \texorpdfstring{$\alpha$}{alpha}}

The considered test case is the square domain partitioned into $256\times 256$ Cartesian elements, with $k=1$. The tolerance is set to $\varepsilon=10^{-14}$.
Let us assess how the preconditioner improves the speed of convergence. \Cref{fig:precond_res_decay} plots the decay of the algebraic residual $||\mathbf{r}||_2/||\mathbf{b}||_2$ without preconditioner (in red) and with preconditioner (in blue), considering increasing values of $\alpha$. 
One can confirm that, as expected, the higher the value of $\alpha$, the better the convergence rate. However, increasing $\alpha$ increases the setup cost of the preconditioner. 
Indeed, (i) the Laplacian problems are solved in larger subdomains, (ii) the approximate matrix $\widetilde{\mathbf{L}}_h$ is denser, which increases the cost of its factorization. 
\Cref{fig:precond_cost} presents the computational cost, measured in CPU time, to reach the algebraic solution. 
One can see that a trade-off between cost and convergence rate must be made to achieve optimal performance. 
If raising $\alpha$ is a good strategy up to $\alpha=6$, above that threshold, the better convergence rate does not make up for the cost of the setup phase, hence making the overall CPU time slightly increase.
We stress that other choices of $k$ yield the same qualitative results.

\begin{figure}
  \begin{center}
  %
  %
  	\begin{subfigure}[b]{0.60\textwidth} 
    \begin{tikzpicture}[scale=0.9]
      \begin{semilogyaxis}[
          width=6.8cm,
		  height=5cm,
          ymajorgrids,
          xlabel={Iterations},
          ylabel={residual},
          xmin=0, xmax=100,
          ymin=1e-14, ymax=1e1,
          ytick={1e-2,1e-6,1e-10,1e-14},
          cycle list name=preconditioner_plot,
          legend cell align=left,
          legend pos=outer north east,
        ]
        \addplot table[x=Iter,y=No,col sep=comma] {data/prec_square_cart_k1_n256.csv};
        \addplot table[x=Iter,y=P2,col sep=comma] {data/prec_square_cart_k1_n256.csv};
        \addplot table[x=Iter,y=P4,col sep=comma] {data/prec_square_cart_k1_n256.csv};
        \addplot table[x=Iter,y=P6,col sep=comma] {data/prec_square_cart_k1_n256.csv};
        \addplot table[x=Iter,y=P8,col sep=comma] {data/prec_square_cart_k1_n256.csv};
        \addplot table[x=Iter,y=P10,col sep=comma] {data/prec_square_cart_k1_n256.csv};
        \legend{unprec. FCG, PFCG($\alpha = 2$), PFCG($\alpha = 4$), PFCG($\alpha = 6$), PFCG($\alpha = 8$), PFCG($\alpha = 10$)};
      \end{semilogyaxis}
      \end{tikzpicture}
      \caption{Decay of the algebraic residual}
        \label{fig:precond_res_decay}
    \end{subfigure}
    \hfill
    \begin{subfigure}[b]{0.39\textwidth} 
        \centering
        \begin{tikzpicture}[scale=1]
            \begin{axis}[
                ybar stacked,
                xlabel=$\alpha$,
                ylabel=CPU time (s),
                bar shift=0pt,
                width=5cm,
                xtick={2,4,6,8,10},
                xmin=0,
                xmax=12,
                ymin=0,
                reverse legend,
                legend cell align={left},
                legend pos=outer north east
            ]
            \addplot coordinates {(2,   0.640) (4,   2.265) (6,   6.468) (8, 15.937) (10, 27.781)}; 
            \addplot coordinates {(2, 151.765) (4, 123.718) (6, 100.765) (8, 93.218) (10, 82.515)}; 
            \legend{setup, iterations};
            \end{axis}
        \end{tikzpicture}
        \caption{CPU time vs. $\alpha$}
        \label{fig:precond_cost}
    \end{subfigure}
    \caption{(a) compares the speed of convergence of the (P)FCG w.r.t. the normalized algebraic residual $||\mathbf{r}||_2/||\mathbf{b}||_2$. (b) compares the CPU time consumed to reach a tolerance of $10^{-14}$, in function of $\alpha$. The test case is the square domain partitioned into $256\times 256$ Cartesian elements, $k=1$.}
    \label{fig:precond_perf}
  \end{center}
\end{figure}
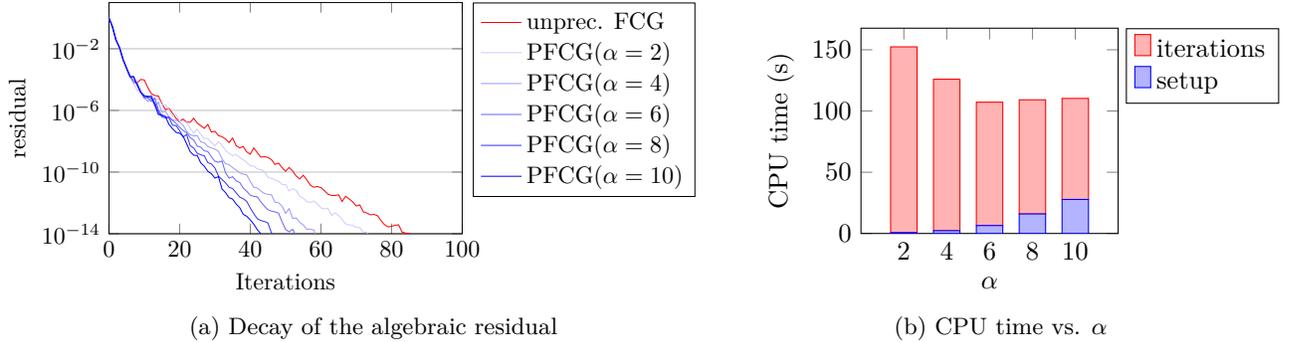

\subsubsection{PFCG convergence rate: \texorpdfstring{$k$}{k}-independence and \texorpdfstring{$h$}{h}-dependency assessment}

Problem \eqref{eq:mixed} on the unit square with exact solution $\psi(x, y) = x\sin(\pi y)e^{-xy}$ is solved on a sequence of 2D grids composed of $N^2$ Cartesian elements, $N\in\{32, 64, 128, 256, 512\}$.
\Cref{tbl:iterations_square} presents, for $k\in\{0,1,2,3\}$, the number of PFCG iterations executed to reach the convergence criterion with $\varepsilon=10^{-8}$. The parameter of the preconditioner is set to $\alpha=8$. For comparison, the number of iterations without preconditioning is reported in brackets.
Firstly, one can remark that the use of the preconditioner allows a convergence rate independent of $k$, while a dependency is observed without preconditioner.
Secondly, we approximate the asymptotic convergence rate with respect to the problem size by using the number of iterations measured on the two finer meshes. 
This yields a dependency in $\mathcal{O}(h^{-0.41})$. 
Additionally, \Cref{tbl:times_square} reports the CPU times of the setup and iteration phases.
The setup includes the Cholesky factorization of the Laplacian matrix and the assembly of the preconditioning matrix. 
In brackets, the same information without preconditioning is displayed.
Based on these results, one can see that the compensation of setup cost by the gain in iteration time is more and more efficient as the problem grows larger.

The equivalent experiment is performed in 3D (exact solution $\psi(x, y) = xz\sin(\pi y)e^{-xy}$) with unstructured tetrahedral meshes.
As the preconditioner is costlier in 3D, we choose $\alpha = 2$.
The number of PFCG iterations (independent of $k$) is presented in \Cref{tbl:iterations_cube:tetra}, which exhibits a dependency in $\mathcal{O}(h^{-0.45})$. 
The CPU times of setup and iteration phases are reported for $k=0$.
In brackets, the same information without preconditioning is displayed.


\begin{table}
    \centering
	\begin{tabular}{r|ccccc}
    	\toprule
    	   \multicolumn{1}{r}{$N=$}  & 32 & 64 & 128 & 256 & 512 \\
    	\midrule
    	   \multicolumn{1}{r}{$k=0$} & 13 (19) & 19 (25) & 26 (33) & 33 (42) & 44 (56) \\
    	   \multicolumn{1}{r}{$k=1$} & 13 (30) & 19 (36) & 26 (47) & 33 (58) & 44 (75) \\
    	   \multicolumn{1}{r}{$k=2$} & 13 (37) & 19 (47) & 26 (62) & 33 (77) & 44 (101) \\
    	   \multicolumn{1}{r}{$k=3$} & 13 (44) & 19 (58) & 26 (75) & 33 (96) & 44 (122) \\
    	\bottomrule
    \end{tabular}
    \caption{Test case: $N^2$ Cartesian elements, $\varepsilon=10^{-8}$. Number of PFCG($\alpha=8$) iterations. In brackets, number of unpreconditioned FCG iterations.}
    \label{tbl:iterations_square}
\end{table}

\begin{table}
    \centering
	\begin{tabular}{rl|rr|rr|rr|rr|rr}
    	\toprule
    	   \multicolumn{2}{r}{$N=$}& \multicolumn{2}{c}{32} & \multicolumn{2}{c}{64} & \multicolumn{2}{c}{128} & \multicolumn{2}{c}{256} & \multicolumn{2}{c}{512} \\
    	\midrule
              \multirow{2}{*}{$k=0$}& setup &  0.4 & (0.0) &  0.9 &  (0.0) &  1.9 &   (0.0) &   3.9 &    (0.0) &   10.7 &    (2.0)\\
                                    & iter. &  0.1 & (0.1) &  0.3 &  (0.4) &  1.8 &   (2.3) &   9.4 &   (11.8) &   50.8 &   (66.7)\\ \hline 
              \multirow{2}{*}{$k=1$}& setup &  1.2 & (0.0) &  2.8 &  (0.0) &  5.7 &   (0.2) &  13.9 &    (1.7) &   41.0 &   (17.2)\\
                                    & iter. &  0.2 & (0.5) &  1.1 &  (2.2) &  6.3 &  (11.5) &  34.1 &   (60.1) &  194.9 &  (336.3)\\ \hline       
              \multirow{2}{*}{$k=2$}& setup &  4.8 & (0.0) & 10.1 &  (0.0) & 22.0 &   (0.6) &  51.0 &    (5.6) &  147.8 &   (57.0)\\
                                    & iter. &  0.6 & (1.7) &  3.5 &  (8.7) & 20.1 &  (48.6) & 109.1 &  (251.9) &  646.2 & (1450.4)\\ \hline       
              \multirow{2}{*}{$k=3$}& setup & 18.2 & (0.0) & 40.5 &  (0.1) & 86.8 &   (1.2) & 188.4 &   (13.4) &  508.0 &  (155.7)\\
                                    & iter. &  1.8 & (6.3) & 10.8 & (33.2) & 61.4 & (177.3) & 319.4 & (1028.9) & 2028.4 & (6271.7)\\
    	\bottomrule
    \end{tabular}
    \caption{Test case: $N^2$ Cartesian elements, $\varepsilon=10^{-8}$. CPU times (in seconds) of the setup and iteration phases. In brackets: analogous quantity without preconditioning.}
    \label{tbl:times_square}
\end{table}


\begin{table}[h]
    \centering
	\begin{tabular}{r|ccccc}
    	\toprule
    	   \multicolumn{1}{r}{$h\approx$}  & $h_0$ & $h_0/2$ & $h_0/4$ & $h_0/8$ \\
    	   \multicolumn{1}{r}{Elements}  & \num{3373} & \num{22869} & \num{162167} & \num{1224468} \\
    	\midrule
              \multicolumn{1}{r}{iter. counts} & 14  (29)  & 20   (40)  & 27   (50)   &  37   (68)    \\
    	   \multicolumn{1}{r}{setup time}   & 5.0 (0.1) & 21.8 (1.1) & 82.8  (8.6) & 581.3 (64.2)  \\
    	   \multicolumn{1}{r}{iter. time}   & 5.3 (0.8) & 24.4 (7.1) & 41.1 (74.9) & 759.4 (1362.1)\\
    	\bottomrule
    \end{tabular}
    \caption{Test case: Cubic domain, unstructured tetrahedral mesh, $\varepsilon=10^{-8}$. Number of PFCG($\alpha=2$) iterations, along with the setup and iteration CPU times (in seconds) for $k=0$. In brackets: analogous quantity without preconditioning.}
    \label{tbl:iterations_cube:tetra}
\end{table}

\clearpage
\section{Conclusion}

In this work, we extended the scheme of Glowinski \emph{et al.}\ to HHO discretizations, yielding an iterative method for the mixed solution of the biharmonic equation. 
Its main advantage lies in the fact that it can be implemented from an existing diffusion code with limited development costs. Namely, notwithstanding the PFCG solver, one requires only the implementation of \eqref{eq:lh_definition} and the right-hand side of \eqref{eq:discrete_subproblem2}. Additionally, for large problems, it allows the use of fast, elliptic solvers for an enhanced time to solution.
Future work will focus on the a priori error analysis of the proposed method and on improving the preconditioner, in order to obtain a convergence rate of the PFCG that is independent of the mesh size, and achieve a scalable behaviour.

\bibliographystyle{plainurl}
\bibliography{references}

\begin{thebibliography}{10}

\bibitem{blaheta_gpcg-generalized_2002}
R.~Blaheta.
\newblock {GPCG}-generalized preconditioned {CG} method and its use with
  non-linear and non-symmetric displacement decomposition preconditioners.
\newblock {\em Numer. Linear Algebra Appl.}, 9(6-7):527--550, 2002.
\newblock \href {https://doi.org/10.1002/nla.295} {\path{doi:10.1002/nla.295}}.

\bibitem{bonaldi_hybrid_2018}
F.~Bonaldi, D.~A. Di~Pietro, G.~Geymonat, and F.~Krasucki.
\newblock A {Hybrid} {High}-{Order} method for {Kirchhoff}–{Love} plate
  bending problems.
\newblock {\em ESAIM: Math. Model. Numer. Anal.}, 52(2):393--421, 2018.
\newblock \href {https://doi.org/10.1051/m2an/2017065}
  {\path{doi:10.1051/m2an/2017065}}.

\bibitem{bouwmeester_nonsymmetric_2015}
H.~Bouwmeester, A.~Dougherty, and A.~V. Knyazev.
\newblock Nonsymmetric {Preconditioning} for {Conjugate} {Gradient} and
  {Steepest} {Descent} {Methods} 1.
\newblock {\em Procedia Computer Science}, 51:276--285, 2015.
\newblock \href {https://doi.org/10.1016/j.procs.2015.05.241}
  {\path{doi:10.1016/j.procs.2015.05.241}}.

\bibitem{chave_hybrid_2016}
F.~Chave, D.~A. Di~Pietro, F.~Marche, and F.~Pigeonneau.
\newblock A {Hybrid} {High}-{Order} method for the {Cahn}-{Hilliard} problem in
  mixed form.
\newblock {\em SIAM J. Numer. Anal.}, 54(3):1873--1898, 2016.
\newblock \href {https://doi.org/10.1137/15M1041055}
  {\path{doi:10.1137/15M1041055}}.

\bibitem{ciarlet_mixed_1974}
P.~G. Ciarlet and P.-A. Raviart.
\newblock A mixed finite element method for the biharmonic equation.
\newblock In {\em Mathematical aspects of finite elements in partial
  differential equations ({Proc}. {Sympos}., {Math}. {Res}. {Center}, {Univ}.
  {Wisconsin}, {Madison}, {Wis}., 1974)}, pages 125--145. Publication No. 33,
  1974.

\bibitem{CIARLET1975277}
P.G. Ciarlet and R.~Glowinski.
\newblock Dual iterative techniques for solving a finite element approximation
  of the biharmonic equation.
\newblock {\em Comput. Methods Appl. Mech. Eng.}, 5(3):277--295, 1975.
\newblock \href {https://doi.org/10.1016/0045-7825(75)90002-X}
  {\path{doi:10.1016/0045-7825(75)90002-X}}.

\bibitem{cockburn_bridging_2016}
B.~Cockburn, D.~A. Di~Pietro, and A.~Ern.
\newblock Bridging the hybrid high-order and hybridizable discontinuous
  {Galerkin} methods.
\newblock {\em ESAIM: Math. Model. Numer. Anal.}, 50(3):635--650, 2016.
\newblock \href {https://doi.org/10.1051/m2an/2015051}
  {\path{doi:10.1051/m2an/2015051}}.

\bibitem{di_pietro_hybrid_2020}
D.~A. Di~Pietro and J.~Droniou.
\newblock {\em The {Hybrid High-Order} method for polytopal meshes}.
\newblock Number~19 in Modeling, Simulation and Application. Springer
  International Publishing, 2020.
\newblock \href {https://doi.org/10.1007/978-3-030-37203-3}
  {\path{doi:10.1007/978-3-030-37203-3}}.

\bibitem{di_pietro_hybrid_2015}
D.~A. Di~Pietro and A.~Ern.
\newblock A hybrid high-order locking-free method for linear elasticity on
  general meshes.
\newblock {\em Comput. Methods Appl. Mech. Eng.}, 283:1--21, 2015.
\newblock \href {https://doi.org/10.1016/j.cma.2014.09.009}
  {\path{doi:10.1016/j.cma.2014.09.009}}.

\bibitem{di_pietro_lemaire_2014}
D.~A. Di~Pietro, A.~Ern, and S.~Lemaire.
\newblock An arbitrary-order and compact-stencil discretization of diffusion on
  general meshes based on local reconstruction operators.
\newblock {\em Comput. Meth. Appl. Math.}, 14(4):461--472, 2014.
\newblock \href {https://doi.org/10.1515/cmam-2014-0018}
  {\path{doi:10.1515/cmam-2014-0018}}.

\bibitem{di_pietro_algebraic_2021}
D.~A. Di~Pietro, F.~Hülsemann, P.~Matalon, P.~Mycek, and U.~Rüde.
\newblock Algebraic multigrid preconditioner for statically condensed systems
  arising from lowest-order hybrid discretizations.
\newblock {\em SIAM J. Sci. Comput.}, 45(4):S329--S350, 2023.
\newblock \href {https://doi.org/10.1137/21M1429849}
  {\path{doi:10.1137/21M1429849}}.

\bibitem{matalon_h-multigrid_2021}
D.~A. Di~Pietro, F.~Hülsemann, P.~Matalon, P.~Mycek, U.~Rüde, and D.~Ruiz.
\newblock An $h$-multigrid method for {Hybrid} {High}-{Order} discretizations.
\newblock {\em SIAM J. Sci. Comput.}, 43(5):S839--S861, 2021.
\newblock \href {https://doi.org/10.1137/20M1342471}
  {\path{doi:10.1137/20M1342471}}.

\bibitem{di_pietro_towards_2021}
D.~A. Di~Pietro, F.~Hülsemann, P.~Matalon, P.~Mycek, U.~Rüde, and D.~Ruiz.
\newblock Towards robust, fast solutions of elliptic equations on complex
  domains through hybrid high‐order discretizations and non‐nested
  multigrid methods.
\newblock {\em Int. J. Numer. Methods Eng.}, 122(22):6576--6595, 2021.
\newblock \href {https://doi.org/10.1002/nme.6803}
  {\path{doi:10.1002/nme.6803}}.

\bibitem{di_pietro_high_order_2022}
D.~A. Di~Pietro, P.~Matalon, P.~Mycek, and U.~Rüde.
\newblock High-order multigrid strategies for hho discretizations of elliptic
  equations.
\newblock {\em Numer. Linear Algebra Appl.}, 30(1):e2456, 2023.
\newblock \href {https://doi.org/10.1002/nla.2456}
  {\path{doi:10.1002/nla.2456}}.

\bibitem{dong_hybrid_2021}
Z.~Dong and A.~Ern.
\newblock Hybrid high-order method for singularly perturbed fourth-order
  problems on curved domains.
\newblock {\em ESAIM: Math. Model. Numer. Anal.}, 55(6):3091--3114, 2021.
\newblock \href {https://doi.org/10.1051/m2an/2021081}
  {\path{doi:10.1051/m2an/2021081}}.

\bibitem{dong_hybrid_2022_2}
Z.~Dong and A.~Ern.
\newblock {$C^0$}-hybrid high-order methods for biharmonic problems, 2022.
\newblock \href {https://doi.org/10.48550/ARXIV.2206.07074}
  {\path{doi:10.48550/ARXIV.2206.07074}}.

\bibitem{dong_hybrid_2022}
Z.~Dong and A.~Ern.
\newblock Hybrid {High}-{Order} and {Weak} {Galerkin} {Methods} for the
  {Biharmonic} {Problem}.
\newblock {\em SIAM J. Numer. Anal.}, 60(5):2626--2656, 2022.
\newblock \href {https://doi.org/10.1137/21M1408555}
  {\path{doi:10.1137/21M1408555}}.

\bibitem{falk_approximation_1978}
Richard~S. Falk.
\newblock Approximation of the {Biharmonic} {Equation} by a {Mixed} {Finite}
  {Element} {Method}.
\newblock {\em SIAM J. Numer. Anal.}, 15(3):556--567, 1978.
\newblock \href {https://doi.org/10.1137/0715036} {\path{doi:10.1137/0715036}}.

\bibitem{glowinski_numerical_1979}
R.~Glowinski and O.~Pironneau.
\newblock Numerical {Methods} for the {First} {Biharmonic} {Equation} and for
  the {Two}-{Dimensional} {Stokes} {Problem}.
\newblock {\em SIAM Review}, 21(2):167--212, 1979.
\newblock \href {https://doi.org/10.1137/1021028} {\path{doi:10.1137/1021028}}.

\bibitem{gustafsson_preconditioned_1984}
I.~Gustafsson.
\newblock A {Preconditioned} {Iterative} {Method} for the {Solution} of the
  {Biharmonic} {Problem}.
\newblock {\em IMA Journal of Numerical Analysis}, 4(1):55--67, 1984.
\newblock \href {https://doi.org/10.1093/imanum/4.1.55}
  {\path{doi:10.1093/imanum/4.1.55}}.

\bibitem{melenk_quasi_2012}
J.~M. Melenk and B.~Wohlmuth.
\newblock Quasi-optimal approximation of surface based lagrange multipliers in
  finite element methods.
\newblock {\em SIAM J. Numer. Anal.}, 50(4):2064--2087, 2012.
\newblock \href {https://doi.org/10.1137/110832999}
  {\path{doi:10.1137/110832999}}.

\bibitem{mihajlovic_efficient_2004}
Milan~D. Mihajlović and David~J. Silvester.
\newblock Efficient parallel solvers for the biharmonic equation.
\newblock {\em Parallel Computing}, 30(1):35--55, 2004.
\newblock \href {https://doi.org/10.1016/j.parco.2003.07.002}
  {\path{doi:10.1016/j.parco.2003.07.002}}.

\bibitem{peisker_numerical_1988}
P.~Peisker.
\newblock On the numerical solution of the first biharmonic equation.
\newblock {\em ESAIM: Math. Model. Numer. Anal.}, 22(4):655--676, 1988.
\newblock \href {https://doi.org/10.1051/m2an/1988220406551}
  {\path{doi:10.1051/m2an/1988220406551}}.

\bibitem{pfefferer_normalder_2019}
J.~Pfefferer and M.~Winkler.
\newblock Finite element error estimates for normal derivatives on boundary
  concentrated meshes.
\newblock {\em SIAM J. Numer. Anal.}, 57(5):2043--2073, 2019.
\newblock \href {https://doi.org/10.1137/18M1181341}
  {\path{doi:10.1137/18M1181341}}.

\bibitem{spigler_theory_1991}
A.~Quarteroni and A.~Valli.
\newblock Theory and {Application} of {Steklov}-{Poincaré} {Operators} for
  {Boundary}-{Value} {Problems}.
\newblock In Renato Spigler, editor, {\em Applied and {Industrial}
  {Mathematics}}, pages 179--203. Springer Netherlands, Dordrecht, 1991.
\newblock \href {https://doi.org/10.1007/978-94-009-1908-2_14}
  {\path{doi:10.1007/978-94-009-1908-2_14}}.

\bibitem{van_gijzen_conjugate_1995}
M.~B. van Gijzen.
\newblock Conjugate {Gradient}-like solution algorithms for the mixed finite
  element approximation of the biharmonic equation, applied to plate bending
  problems.
\newblock {\em Comput. Methods Appl. Mech. Eng.}, 121(1):121--136, 1995.
\newblock \href {https://doi.org/10.1016/0045-7825(94)00737-8}
  {\path{doi:10.1016/0045-7825(94)00737-8}}.

\end{thebibliography}

\appendix

\section{Proof of \texorpdfstring{\cref{lem:link_condensed_hybrid}}{Lemma 1}} \label{annex}

Let $v_{\faces} \in \brokenPolyF{h}{k}$ and $\underline{w}_h \in \hybridSpace$. 
For all $T\in \cells$, $\rcvT v_{\facesT}$ is by definition solution of \eqref{eq:local_subproblem_1}. By linearity of $a_h$, combining both terms of \eqref{eq:local_subproblem_1} yields
\begin{equation*} 
    a_T(\underline{\Theta}_T v_{\facesT}, (w_T, 0)) = 0.
\end{equation*}
That being true for all $T\in \cells$ proves the global counterpart 
\begin{equation} \label{eq:local_solution_property_global}
a_h(\underline{\Theta}_h v_{\faces}, (w_{\cells}, 0)) = 0.
\end{equation}
By splitting $\underline{w}_h$ into $(w_{\cells}, 0) + (0, w_{\faces})$ and by linearity of $a_h$, we have
\begin{equation*}
a_h(\underline{\Theta}_h v_{\faces}, \underline{w}_h) = a_h(\underline{\Theta}_h v_{\faces}, (w_{\cells}, 0)) + a_h(\underline{\Theta}_h v_{\faces}, (0, w_{\faces})).
\end{equation*}
Property \eqref{eq:local_solution_property_global} allows to cancel the first term of the right-hand side.
Then, by writing $(0, w_{\faces}) = (\rcvtf w_{\faces}, w_{\faces}) - (\rcvtf w_{\faces}, 0) = \underline{\Theta}_h w_{\faces} - (\rcvtf w_{\faces}, 0)$ and using the linearity of $a_h$, the second term rewrites
\begin{equation} \label{eq:lem1:temp1}
    a_h(\underline{\Theta}_h v_{\faces}, (0, w_{\faces})) 
    =
    a_h(\underline{\Theta}_h v_{\faces}, \underline{\Theta}_h w_{\faces})
    -
    a_h(\underline{\Theta}_h v_{\faces}, (\rcvtf w_{\faces}, 0)).
\end{equation}
Using \eqref{eq:local_solution_property_global} again to cancel the last term of \eqref{eq:lem1:temp1} finishes the proof of \eqref{eq:link_condensed_hybrid}.

\end{document}